\documentclass{article}
\usepackage[utf8]{inputenc}

\usepackage{geometry}
\usepackage{eufrak}

\usepackage{appendix}
\usepackage{amsfonts, amssymb, amsmath, stmaryrd, amsthm, accents, bbm}
\usepackage{mathtools}
\numberwithin{equation}{section}
\usepackage[linesnumbered,ruled,vlined]{algorithm2e}
\usepackage{enumitem}
\usepackage{caption2}

\usepackage[pdftex,
colorlinks=true, 
linkcolor=blue, 
citecolor=blue, 
pagebackref=false, 
]{hyperref}

\usepackage[version=4,arrows=pgf-filled,
textfontname=sffamily,
mathfontname=mathsf]{mhchem}
\usepackage{braket}
\usepackage{comment}
\usepackage[normalem]{ulem}
\usepackage{xcolor}
\usepackage[colorinlistoftodos]{todonotes}

\usepackage[official]{eurosym}
\usepackage{booktabs}
\usepackage{tabularx}
\usepackage{makecell}
\usepackage{array}

\newcolumntype{C}[1]{>{\centering\arraybackslash}m{#1}}
\newcommand{\celltop}{\rule{0pt}{5.5ex}}

\newtheorem{proposition}{Proposition}[section]
\newtheorem{corollary}{Corollary}[section]
\newtheorem{definition}{Definition}[section]
\newtheorem{theorem}{Theorem}[section]
\newtheorem{lemma}{Lemma}[section]
\newtheorem{remark}{Remark}[section]
\newtheorem{assumption}{Assumption}[section]

\newcommand{\norm}[1]{\left\Vert #1 \right\Vert}
\newcommand{\abs}[1]{\left\vert #1 \right\vert}
\newcommand{\ind}{\mathbbm{1}}
\newcommand{\R}{\mathbb{R}}
\newcommand{\N}{\mathbb{N}}
\newcommand{\E}{\mathbb{E}}

\newcommand{\bP}{\mathbb{P}}

\newcommand{\cP}{\mathcal{P}}
\newcommand{\cT}{\mathcal{T}}
\newcommand{\cA}{\mathcal{A}}
\newcommand{\cF}{\mathcal{F}}
\newcommand{\cG}{\mathcal{G}}
\newcommand{\cL}{\mathcal{L}}

\title{Non-Potential Congestion Games with a Continuum of Players} 

\author{Alix Dupont\footnote{EDF R\&D, 91120 Palaiseau, France. E-mail: \href{alix.dupont@edf.fr}{alix.dupont@edf.fr}.}, Quentin Petit\footnote{EDF R\&D, 91120 Palaiseau, France and Finance for Energy Market Research Centre (FIME Lab). E-mail: \href{quentin.petit@edf.fr}{quentin.petit@edf.fr}.}, and Adrien Séguret\footnote{EDF R\&D, 91120 Palaiseau, France and Finance for Energy Market Research Centre (FIME Lab). E-mail: \href{adrien.seguret@edf.fr}{adrien.seguret@edf.fr}.}}

\begin{document}

\maketitle


\begin{abstract}
We study a class of congestion games with a continuum of heterogeneous players and a finite set of resources, without assuming the existence of a potential function. Players differ in their individual characteristics and may contribute differently to congestion, which is represented through general aggregate features of the population choosing each resource. We characterize the conditions under which such a game admits a potential and show that this structure generally fails when players have heterogeneous impacts on congestion. To analyze the non-potential case, we introduce the concept of Stochastic Variational Nash Equilibrium, in which players observe idiosyncratic cost shocks before selecting a resource. We show that equilibrium choices are pure conditional on players' characteristics and shock realizations, whereas the induced aggregate allocation is probabilistic. Under Gumbel-distributed shocks, the equilibrium is equivalent to a continuum-player logit Quantal Response Equilibrium. Using a fixed-point formulation and the Schauder--Tychonoff theorem, we establish existence for every positive noise level and recover a pure Nash equilibrium in the zero-noise limit. Under a coordinate-wise monotonicity condition on aggregate interaction costs, we further prove uniqueness and convergence of the regularized allocations to the pure equilibrium. An application to electric vehicle charging illustrates how heterogeneous energy demands jointly affect station choice, congestion, and market shares.
\end{abstract}


\section{Introduction}
\label{sec:introduction}

\subsection{General introduction}
Congestion games, which model strategic interactions among players competing for 
shared resources, are a foundational topic in game theory. Their study in settings 
with a large, atomless population of players has a rich history, originating in the 
economic analysis of markets with a continuum of traders \cite{Aumann1964,Aumann1966} 
and placed on rigorous game-theoretic foundations by Schmeidler \cite{schmeidler1973equilibrium}, 
who established the existence of pure-strategy Nash equilibria in nonatomic games. 
This framework was subsequently developed in the context of congestion games by 
Milchtaich \cite{milchtaich2000generic,milchtaich2004social}, who studied equilibrium 
uniqueness and social optimality in nonatomic crowding games. The resulting continuum framework is particularly well-suited for applications 
in urban transportation, the canonical example being the Wardrop equilibrium 
principle \cite{wardrop1952road}, as well as in telecommunications, 
economics, or environmental science.

A significant portion of the literature on congestion games has focused on the 
class of potential games. The use of a potential function 
setting was pioneered by Beckmann, McGuire and Winsten \cite{beckmann1956studies} 
in the context of traffic equilibrium. The concept was later given a general treatment by Monderer and Shapley \cite{monderer1996potential}, who introduced 
the term \emph{potential game} and characterized this class of games. In all these settings, the existence of a potential function, whose optima correspond 
to Nash equilibria, transforms the strategic problem into a variational 
one. This powerful 
property has been exploited extensively. In the realm of nonatomic games, Crippa, Jimenez, and Pratelli \cite{crippa2009optimum} use a variational 
approach to analyze a model with queueing effects where congestion at a resource depends solely on 
the total mass of players choosing this resource, a setting that falls squarely within the potential game 
framework. However, the conditions for a game to admit a potential are restrictive. 
The work of Harks, Klimm, and M\"{o}hring \cite{harks2011characterizing} provides a 
sharp characterization for discrete weighted congestion games, demonstrating that even 
simple forms of player-specific heterogeneity can break the potential structure. This 
suggests that in many realistic scenarios, where players contribute non-uniformly to 
congestion, optimisation-based methods are inapplicable.

In this paper, we address this limitation by studying a static congestion game with 
a continuum of heterogeneous players facing a discrete choice, without assuming that 
the game admits a potential. We consider a population of atomless players distributed 
over a characteristics space $X$, where the interaction cost $h_i(a_i)$ associated 
to resource $i$ can depend on complex aggregate measures $a_i$ defined on $X$ representing the proportion of players choosing $i$ depending on their characteristics, such as arbitrary moments. This formulation 
allows for rich heterogeneity in how players contribute to congestion, and strictly 
generalizes the potential setting of \cite{crippa2009optimum}. While player heterogeneity is considered in \cite{blanchet2016optimal}, the players' impact on congestion remains homogeneous in their setting. To handle more complex interactions, \cite{Jacquot2022} introduced the notion of \textit{Variational Wardrop Equilibrium} to characterize Wardrop equilibria for non-atomic games with heterogeneous players and coupling aggregative constraints. Furthermore, characterizing the Nash equilibrium of a non-atomic population is frequently studied to solve the lower level of bilevel optimization problems, with applications in optimal pricing models \cite{carlier2018optimal}, optimal allocation \cite{guerra2026optimal}, and spatial equilibria \cite{achdou2023simple}. More complex interactions among players in a dynamic setting can be found in Mean Field Game theory (MFG) \cite{lasry2006jeux, lasry2007mean}.

We provide a characterization for 
the existence of a potential in this functional setting, offering a continuous-game 
counterpart to the discrete analysis of \cite{harks2011characterizing}. Our primary 
focus, however, is on the general non-potential case, which requires tools beyond 
optimisation.
To tackle the non-potential setting, we employ a stochastic regularization method. 
This leads to the concept of a \emph{Stochastic Variational Nash Equilibrium} (SVNE), 
which serves as the continuum-player analogue of a Quantal Response Equilibrium (QRE). 
The QRE framework, introduced by McKelvey and Palfrey \cite{mckelvey1995quantal}, 
models boundedly rational players who make noisy best responses, with the probability 
of choosing a strategy increasing with its expected payoff. Although an SVNE and a QRE induce the same conditional choice probabilities, the SVNE provides a richer representation of equilibrium behavior by explicitly describing players’ choices for each realization of the idiosyncratic shocks. We succeed to show that players play in pure strategies upon observing this 
noise, even though they appear mixed from an aggregate perspective. We prove that a 
SVNE is equivalent to a continuum QRE, where a noise parameter $\sigma > 0$ controls 
the degree of irrationality. By assuming that the idiosyncratic noise follows a Gumbel 
distribution, one recovers the well-known logit choice model, a cornerstone of 
discrete choice theory since the work of McFadden \cite{mcfadden1974conditional}. As 
$\sigma \to 0$, players approach perfect rationality and their choices converge to a 
deterministic Nash Equilibrium (NE). Beyond its behavioral interpretation, this 
regularization provides a powerful mathematical device: it ensures the existence of 
a SVNE for any $\sigma > 0$, from which a deterministic NE is then recovered in the 
limit.

Since the game is not assumed to be a potential game, we cannot rely on optimisation 
techniques. Instead, the analysis of the regularized game hinges on fixed-point 
theory. The equilibrium conditions are reformulated as a fixed-point equation 
$a = F(a)$, where $F$ is a nonlinear operator derived from the logit choice 
probabilities. Existence of a fixed-point is established via the  Schauder–Tychonoff 
theorem, which is well-adapted to this infinite-dimensional setting.

Uniqueness of the NE is established under a monotonicity condition on the interaction 
costs. These types of condition are well-known to provide uniqueness \cite{smith1979existence}. Here, we assume that the cost $h_i$ depends only on a finite 
collection of aggregate quantities and is monotone with respect to these aggregates. Under this 
condition, any two equilibria must generate identical aggregate quantities, which in 
turn forces the equilibria to coincide. 

\paragraph{Contributions.}
This paper makes the following contributions to the theory of congestion games with 
a continuum of players.
\begin{itemize}
    \item \textbf{A General Framework for Non-Potential Congestion Games:} We propose 
    and analyze a model that generalizes existing potential-based frameworks 
    \cite{crippa2009optimum} to settings with complex, heterogeneous interactions. We 
    provide a characterization of when the game admits a potential in this continuous 
    setting, complementing the discrete analysis of \cite{harks2011characterizing}.

    \item \textbf{A Noise-Contingent Representation of Equilibrium}: We introduce the SVNE, a stochastic variational formulation that provides a richer description of equilibrium behavior than the standard QRE \cite{mckelvey1995quantal} in the continuum setting.
    By characterizing each player's strategy as a function of both their characteristic and their noise realization, the SVNE explicitly describes noise-contingent individual choices, whereas the QRE retains only the resulting choice probabilities conditional on the player’s characteristic.
    We show that players choose pure strategies almost surely once their noise is realized.
    We then establish the equivalence between the SVNE and the continuum QRE in terms of conditional choice probabilities. Finally, we distinguish the SVNE from the Stochastic User Equilibrium, which is traditionally formulated for graph-based nonatomic routing games \cite{daganzo1977stochastic}.

    \item \textbf{Existence of Equilibrium via SVNE Regularization:} We provide a 
    rigorous proof of the existence of a pure Nash Equilibrium for a broad class of 
    non-potential games. The SVNE is shown to exist for any $\sigma > 0$ via the 
     Schauder–Tychonoff fixed-point theorem, and a deterministic NE is recovered in the 
    limit when $\sigma$ goes to $0$.

    \item \textbf{Uniqueness via a Monotonicity Argument:} We establish a 
    coordinate-wise monotonicity condition on the aggregate interaction costs that is 
    sufficient to guarantee uniqueness of equilibrium.
\end{itemize}

\paragraph{Structure of the paper.}
The rest of the paper is structured as follows. Section \ref{sec:assumptions_main_results} 
introduces the notions of equilibrium used in the paper, the standing assumptions and presents our main results. Section \ref{sec:mathematical_analysis} contains the core mathematical 
analysis: we investigate conditions for the game to be potential, establish the equivalence between the SVNE and the continuum QRE, 
and provide detailed proofs for existence and uniqueness. Finally, Section \ref{illustrative_ex} presents numerical simulations illustrating the model.

\subsection{A general congestion game}\label{congestion_game_def}

Let $X$ denote a bounded domain of $\R^d$, characterizing the set of players' characteristics. Let $\mu$ be the distribution of the players on $X$, with the assumption that $\mu \in L^\infty(X,\R_+) $.

Each individual faces a discrete choice among the set of resource $I:=\{1,\ldots,N\}$ and pays a cost accordingly.
A strategy profile is defined as a mapping $a:X\mapsto \mathbb{R}^N_+$ where $a_i(x)$ represents the proportion of players with $x$ choosing the resource $i$. The set of all strategy profiles is denoted by $\cA$ and defined as:
\begin{equation}\label{def_S}
    \cA :=\left\{a \in L^2(X,[0,1]^N)\,\vert\,\sum_{i\in I}a_i(x)= 1 \quad \text{for }\mu.a.e.\, x\in X\right\}.
\end{equation}
For a player with characteristic $x\in X$, the cost associated with the choice of resource $i \in I$ is:
\begin{equation}\label{eq:sep_cost}
    c_i(x) +h_i(a_i) ,
\end{equation}
where $c_i$ represents the idiosyncratic cost and $h_i$ accounts for the interaction effects (e.g., congestion or externalities).
The aim of this article is to establish existence and uniqueness of equilibrium, as well as providing stability results. 
Throughout this article, for any $a \in \cA$, we adopt the following notation:
    \begin{equation*}
        h(a):=\big(h_1(a_1),\ldots,h_N(a_N)\big).
    \end{equation*}

\begin{remark}\label{rem:first_properties_A}
    The set of strategy profiles $\mathcal A$ is convex, as the constraints
$a_i(x)\ge 0$ and $\sum_{i=1}^N a_i(x)=1$ are preserved pointwise under
convex combinations (the upper bound $a_i\le 1$ is implied by these two
constraints). Finally, $\mathcal A$ is closed and bounded in
$L^2(X,\mu;\mathbb R^N)$, hence weakly compact. 
\end{remark}

\subsection{Severals notions of equilibrium}\label{sec:notions_equilibrium}

To establish existence and uniqueness results for the game's equilibrium, we will use two notions of equilibrium: the Pure Nash Equilibrium (PNE), and the Stochastic Variational Nash Equilibrium (SVNE$_\sigma$) for a given regularization parameter $\sigma \ge 0$. 

We will begin by demonstrating the existence of an SVNE$_\sigma$ for $\sigma > 0$. By taking the limit as $\sigma \to 0$, we will then deduce the existence of a SVNE$_0$. Finally, we will show that under an additional assumption (related to a Voronoi-type partition), the existence of a SVNE$_0$ guarantees that of a PNE. This approach is also advantageous for the uniqueness analysis: if uniqueness is established for the SVNE$_0$, it is directly inherited by the PNE.

Here is the formal definitions of the different notions of equilibrium:
\begin{definition}[Pure Nash Equilibrium]\label{nash_eq}
We say that a partition $(A_1,...,A_N)\in \cP(X)^N$ is a Pure Nash Equilibrium (PNE) if for every $i\in\{1,\ldots,N\}$ it holds:
\begin{equation*}
    \displaystyle   A_i=    \displaystyle \left\{ x\in X \,\big\vert\,c_i(x)+ h_i(\ind_{A_i}) = \min_{k=1}^N c_k(x)+ h_k(\ind_{A_k})\right\}.
\end{equation*}
\end{definition}
\begin{definition}[Stochastic Variational Nash Equilibrium] \label{def:SVNE} Let $(\Omega,\mathcal{G}, \bP)$ be a probability space and $(\mathcal{E}_i)_{i=1,...,N}$ be i.i.d random variables from $\Omega$ into $\R^N$ of law $\nu$ . Let us define:
    \begin{displaymath}
        \mathcal F = \left\{f:\R^{N}\times X\rightarrow [0,1]^N\;\vert\;f\text{ is measurable and }\sum_{i\in I} f_i(\varepsilon,x) = 1\right\}.
    \end{displaymath} 
    Let $\sigma\geq0$, an element $f$ belonging to $\mathcal F$ is a Stochastic Variational Nash equilibrium (SVNE$_\sigma$), if for any $g\in \mathcal F$:
        \begin{equation}\label{eq:svar_nash_eq_def}
    \begin{array}{cc}
&    \displaystyle \sum_{i\in I}\int_X\int_{\R^N} (c_i(x)+h_i(a_i)-\sigma\varepsilon_i)f_i(\varepsilon,x)d\nu^N(\varepsilon)d\mu(x)\\
    & \displaystyle\quad\le \sum_{i\in I}\int_X \int_{\R^N}(c_i(x)+h_i(a_i)-\sigma \varepsilon_i)g_i(\varepsilon,x)d\nu^N(\varepsilon)d\mu(x),
\end{array}
\end{equation}
where $\nu^N(\varepsilon) = \prod_{i=1}^N\nu(\varepsilon_i)$ and for any $(x,i)\in X\times I$, $a_i(x)$ is the expectation of $f_i(\mathcal{E},x)$, i.e. 
\begin{equation}
    \label{def_a}
a_i(x) := \E[f_i(\mathcal{E},x)] = \int_{\R^N} f_i(\varepsilon,x)d\nu^N(\varepsilon).
\end{equation}
\end{definition}

Definition \ref{def:SVNE} is related to the notion of QRE. The link will be made clear by Corollary \ref{cor:SVNE_fixed_point}. Furthermore, Proposition \ref{prop:justification_SVNE} shows that an SVNE$_\sigma$ coincides with a Nash equilibrium of the non-atomic game defined on the extended characteristic space $\mathbb{R}^N \times X$, in the sense that almost every player selects a resource minimizing their noise-perturbed cost.

In the above formulation, players can play in mixed strategies. Indeed, at an equilibrium $f$, a player with characteristic $x$ facing the noise $\varepsilon$ choose $i$ with probability $f_i(\varepsilon,x)$. In Section \ref{subsec:prop_svne}, we are able to show that players play in pure strategies.
\begin{remark}
    In the above definition, the map $a$ defined in \eqref{def_a} belongs to $\cA$ defined in \eqref{def_S}.
\end{remark}

\subsection{Notations}

In what follows, given two measures $\mu$ and $\nu$, we denote their product measure by $\mu\otimes\nu$. If $\mu$ and $\nu$ are absolutely continuous with respect to the Lebesgue measures, we use, with a slight abuse of notation, the same symbols $\mu$ and $\nu$ for both the measures and their densities. Under this convention, the density of $\mu\otimes\nu$ is given by $x\mapsto \nu(x)\mu(x)$.
For any set $A$, we denote by $\ind_{A}$ the indicator function of $A$. Let $X$ be a Lebesgue measurable subset of $\mathbb{R}^n$. For any two measurable subsets $A,B\subseteq X$, we denote by $A\mathbin{\triangle}B := (A\setminus B)\cup(B\setminus A) $ their symmetric difference. 
Let $\cP(X)$ be the power set of $X$. A family of measurable sets $(A_1,\ldots,A_N)$ is a partition of $X$ up to $\mu$-null sets if \begin{displaymath} \mu\left(X\setminus\bigcup_{i=1}^N A_i\right)=0 \quad\text{and}\quad \mu(A_i\cap A_j)=0 \quad\text{for all }i\neq j. \end{displaymath}
We denote by $L^\infty(X, \mathbb{R}^m)$ the Banach space of measurable and essentially bounded functions from $X$ to $\mathbb{R}^m$. 
For $1 \leq p <\infty$, we denote by $L^p(X, \mu; \mathbb{R}^m)$ the space of equivalence classes of measurable functions $f: X \to \mathbb{R}^m$ such that the $p$-th power of their norm is integrable w.r.t. the measure $\mu$, i.e., $$ \int_X |f(x)|^p d\mu(x) < \infty, $$ where $|\cdot|$ is the standard Euclidean norm and the integral is with respect to the Lebesgue measure on $X$. In $L^2(X,\mu;\R^m)$, the inner product is
\begin{displaymath}
    \langle u,v\rangle_{L^2}
    :=
    \sum_{i=1}^m
    \int_X u_i(x)v_i(x)\,d\mu(x).
\end{displaymath}
We denote by $L^2(X, \mu)$ the space of equivalence classes of real valued measurable functions $f$ on $X$ that are square-integrable with respect to the measure $\mu$. 

If not said otherwise, the notion of derivative always refers to Fréchet derivative. In this case,
a function $f:X\rightarrow Y$, with $X$ and $Y$ Banach spaces, is said to be of class $C^1$ if it is differentiable and its first derivative, $Df$, is continuous.
Building on this, a function is of class $C^2$ if it is twice differentiable and its second derivative, $D^2f$, is continuous.  If $f: X \times Y \to \mathbb{R}^n$, with $n\geq 1$, is continuously differentiable, we denote by $D_y f$ its partial derivative (or Jacobian matrix) with respect to $y$. If $f$ is scalar-valued ($n=1$), we may also use $\nabla_y f$ to denote its partial gradient vector. Let $X$ be a subset of a Banach space, we call the domain of a map $f:X\rightarrow \R\cup \{+\infty\}$ the set $f^{-1}(\R)$.

\section{Standing assumptions and main results}\label{sec:assumptions_main_results}
This section is devoted to introduce the assumptions used in the rest of the paper, and the main results we aim to establish.

\begin{assumption}[On the transport costs] {\color{white}a}
\label{ass:transport_costs}

\begin{enumerate}
    \item For any $i\in I$, the map $c_i:X\rightarrow \R_+$ is continuous. \label{deriv_c}
    \item (No tie assumption) For any indeces $i,j$ such that $i\neq j$, and any constant $\beta$, the set $\{x\in X \vert c_i(x) - c_j(x)=\beta\}$ is $\mu$-negligible.\label{voronoi_assum}
\end{enumerate}
    
\end{assumption}

\begin{assumption}[On the noise distribution]\label{density_assum}
    The distribution $\nu$ has the density of a standard Gumbel distribution.
\end{assumption}

\begin{assumption}[On the interaction term]\label{ass:existence_interaction_term}
    The interaction term $h:\cA\subset L^2(X,\mu;\R^N)\rightarrow (\R\cup \{+\infty\})^N $
    \begin{enumerate}
        \item
            The domain of $h$ is non empty in $\cA$.
        \item \label{bound_h}
            For any $a\in \cA$ it holds
            \begin{equation*}
                0 \leq   \min_{i\in I}h_i(a_i)< +\infty.
            \end{equation*}
        \item \label{lsc_h}
            For the weak topology in $L^2(X,\mu;[0,1])$ and for any $i\in I$ 
            \begin{itemize}
                \item the map $h_i$ is lower-semicontinuous on $L^2(X,\mu;[0,1])$ ;
                \item $h_i(0)<+\infty$;
                \item the domain of $h_i$ is open in $L^2(X,\mu;[0,1])$ ;
                \item $h_i$ is continuous on its domain.
            \end{itemize}
    \end{enumerate}
\end{assumption}

\begin{remark}
  The requirement in Assumption \ref{ass:existence_interaction_term} that the domain of $h_i$ be open with respect to the weak topology of $L^2(X,\mu;\mathbb{R}^N)$ is rather restrictive. Assumption \ref{ass:uniqueness_interaction_term} provides a particular setting in which this requirement is satisfied.  
\end{remark}

\begin{assumption}[On the interaction term; uniqueness]\label{ass:uniqueness_interaction_term}
The interaction term $h:\cA\subset L^2(X,\mu;\R^N)\rightarrow (\R\cup \{+\infty\})^N $ takes the following form:
\begin{displaymath}
    h(a) =\big(\tilde h_1(\Psi  a_1),\ldots,\tilde h_N(\Psi  a_N)\big),
\end{displaymath}
where the linear map $\Psi:L^2(X,\mu;\R)\rightarrow \R^m$ is defined by:
\begin{equation}\label{def_lambda}
    \Psi f:=\left(\int_X \psi^1(x)f(x)d\mu(x),\ldots,\int_X \psi^m(x)f(x)d\mu(x)\right),
\end{equation}
and where $\tilde h:\R^m\rightarrow (\R\cup \{+\infty\})^N $ is such that:
\begin{enumerate}
    \item \label{pos_lambda}
    For any $\ell$ in $\{1,\ldots,m\}$, the map $\psi^\ell$ introduced in \eqref{def_lambda} is in $L^2(X,\mu;\mathbb{R}^+)$. 
    \item \label{cont_domain_h}
For any $i\in I$, the domain of $\tilde h_i$ is open in $\mathbb{R}^m$ and is continuous over it.
\item \label{decreasing}
For any $i\in I$, $\tilde h_i$ is coordinate-wise non decreasing over $\mathbb{R}^m$, i.e., for any $x,y\in \mathbb{R}^m$, if $x^\ell\leq y^\ell$ for any $\ell\in \{1,\ldots,m\}$, then $\tilde h_i(x)\leq \tilde h_i(y)$.
\end{enumerate}
    
\end{assumption}

\begin{remark}
Note that in Assumption \ref{ass:uniqueness_interaction_term}, we could make $\Psi$ depend on $i\in I$. However, for simplicity of notation, we have assumed that it is independent.
\end{remark}

The following theorems is one of the main results of the paper and states existence and uniqueness of the SVNE$_\sigma$ and PNE.
\begin{theorem}\label{main_result}
 Under Assumptions \ref{ass:transport_costs} and \ref{ass:existence_interaction_term}, there exists a PNE and, under Assumption \ref{density_assum}, a SVNE$_\sigma$ for any $\sigma \geq 0$.
 \smallskip
    \noindent
    If, furthermore, Assumption \ref{ass:uniqueness_interaction_term} holds, then the PNE is unique up to $\mu$-null sets and for any $\sigma\geq 0$ the SVNE$_\sigma$ is unique. 
\end{theorem}
\begin{remark}
The following counter-example illustrates that the coordinate-wise non decreasing assumption of $\tilde{h}$ in Assumption \ref{ass:uniqueness_interaction_term} is crucial to get uniqueness of the PNE. 
Assume $I=\{1,2\}$, $X=[0,1]$, $\mu$ is equal to the Lebesgue measure, $c_1(x):= x$, $c_2(x):=1-x$, $\tilde{h}_1(y)=\tilde{h}_2(y)=- y$ and $\Psi f:=\int_Xf(x)d\mu(x)$. Then for any $s\in [0,1)$, the partition $([0,s],(s,1])$ is a PNE.

\end{remark}

For every $\sigma\geq 0$, let $a^\sigma=(a_1^\sigma,\ldots,a_N^\sigma)$ denote the vector associated with the $\mathrm{SVNE}_\sigma$ introduced in Definition~\ref{def:SVNE}. We define the induced allocation $(B_1^\sigma,\ldots,B_N^\sigma)\in\mathcal{P}^N$ by \begin{equation}\label{def_B_sigma} B_i^\sigma = \left\{ x\in X : c_i(x)+h_i(a_i^\sigma) = \min_{1\leq k\leq N} \bigl(c_k(x)+h_k(a_k^\sigma)\bigr) \right\}, \qquad i\in\{1,\ldots,N\}. \end{equation}
Lemma \ref{lem:intersection_mesure_zero} in Section \ref{subsec:prop_svne} ensures that up to a $\mu$-null set,   $(B_1^\sigma,\ldots,B_N^\sigma)$ is a partition of $X$.
The following theorem establishes the convergence of the allocation induced by the $\mathrm{SVNE}_\sigma$ to the PNE as $\sigma$ tends to zero.

\begin{theorem}\label{thm:conv} Under Assumptions~\ref{ass:transport_costs}-- \ref{ass:uniqueness_interaction_term}, for every $i\in\{1,\ldots,N\}$,
\begin{equation*}
    \mu(B_i^\sigma\triangle A_i)\longrightarrow 0 \qquad\text{as }\sigma\to0,
\end{equation*}
where $(A_1,\ldots,A_N)$ is a PNE. 
\end{theorem}

\section{Mathematical analysis}\label{sec:mathematical_analysis}

This section is organized into four parts. First, Section \ref{subsec:non_potential_game} presents several results demonstrating that the congestion game defined in Section \ref{congestion_game_def} lacks a potential structure. This absence of a potential motivates our alternative framework, which is centered on the $\text{SVNE}_\sigma$. Second, Section \ref{subsec:prop_svne} details the fundamental properties of this equilibrium concept, including its formal connection to the QRE. Third, in Section \ref{equilibria_section}, we establish the existence of equilibria, rather than employing a classical optimization approach, our proof relies on a fixed-point approach. Finally, Section \ref{uniqueness_section} provides a uniqueness result
based on the monotone structure of the interaction cost.

\subsection{A non-potential congestion game}\label{subsec:non_potential_game}

In \cite{crippa2009optimum}, the authors established the existence and
uniqueness of a Nash equilibrium when the congestion only depends on the
mass of players, i.e.\ when for any $i$ the costs are
\begin{equation}\label{eq:h_depends_on_mass}
    h_i(a_i) = \tilde h_i\left(\int_X a_i(x)\,d\mu(x)\right)
\end{equation}
for a function $\tilde h_i:\R\rightarrow\R$. Their approach is based on the
fact that such a game admits a potential. In this paragraph we determine exactly how far this
variational structure extends. In Theorem \ref{thm:path_characterization} below, we characterize the existence of a potential
by a \emph{path-invariance} property of the cost operator, which requires no
differentiability whatsoever. In Proposition~\ref{prop:potential_condition}, we recover a differential symmetry test
as its smooth specialization. The characterization then
yields a sharp negative result: as soon as players contribute unequally to
congestion, the game admits no potential
(Proposition~\ref{prop:aggregate}). In words, the game is a potential
game if and only if players contribute equally to congestion. For the
sake of generality, we do not assume in this section the separability of
the costs as in \eqref{eq:sep_cost}. Accordingly, the cost for a player
with characteristic $x\in X$ choosing a resource $i\in I$ is
$h_i(x,a_i)$, where $h_i$ is defined on
$X\times L^2(X,\mu;\mathbb{R}_+)$ and is real-valued throughout this
section.

\bigskip
We first introduce the geometry used on the feasible set $\mathcal A$.
We consider the affine space
\begin{displaymath}
    \mathcal H :=
    \left\{
        a=(a_1,\ldots,a_N)\in L^2(X,\mu;\mathbb R^N)
        \;\middle|\;
        \sum_{i=1}^N a_i=1 \quad\mu\text{-a.e.}
    \right\},
\end{displaymath}
and its associated linear vector space
\begin{displaymath}
    \mathcal H_0 :=
    \left\{
        b=(b_1,\ldots,b_N)\in L^2(X,\mu;\mathbb R^N)
        \;\middle|\;
        \sum_{i=1}^N b_i=0 \quad\mu\text{-a.e.}
    \right\}.
\end{displaymath}

\begin{remark}[Structural properties of $\mathcal A$]
\label{rem:structure_A}
Note that $\mathcal A\subset\mathcal H$ and, for all
$a,\widetilde a\in\mathcal A$, the increment $\widetilde a-a$ belongs to
$\mathcal H_0$. 
\end{remark}

For $a\in\mathcal A$, we define the cone of feasible directions at $a$ by
\begin{displaymath}
    \mathcal T_{\mathcal A}(a) :=
    \left\{
        b\in\mathcal H_0
        \;\middle|\;
        \text{there exists }\varepsilon>0
        \text{ such that } a+tb\in\mathcal A
        \text{ for every }t\in[0,\varepsilon]
    \right\}.
\end{displaymath}
Following~\cite[Sect.~2.2]{BonnansShapiro2000}, we use the one-sided
directional derivative along feasible directions of the constraint set.

\begin{definition}[Feasible directional derivative]
Let $E$ be a Banach space and let $F:\mathcal A\to E$. For $a\in\mathcal A$
and $b\in\mathcal T_{\mathcal A}(a)$, the feasible directional derivative of
$F$ at $a$ in the direction $b$ is defined by
\begin{displaymath}
    D_{\mathcal A}^{+}F(a)[b]
    :=
    \lim_{t\downarrow0}
    \frac{F(a+tb)-F(a)}{t},
\end{displaymath}
whenever this limit exists in $E$.
\end{definition}

Throughout this section, we assume that the total cost induces a
well-defined operator
\begin{displaymath}
\begin{aligned}
    \boldsymbol h:\mathcal A
    &\longrightarrow L^2(X,\mu;\mathbb R^N),\\
    a
    &\longmapsto h(\cdot,a)
    :=
    \bigl(h_1(\cdot,a_1),\ldots,h_N(\cdot,a_N)\bigr).
\end{aligned}
\end{displaymath}

\begin{definition}[Potential game]
\label{def:potential_games}
The game admits a potential if there exists a functional
$\mathcal G:\mathcal A\to\mathbb R$
such that, for every $a\in\mathcal A$ and every
$b\in\mathcal T_{\mathcal A}(a)$, the feasible directional derivative
$D_{\mathcal A}^{+}\mathcal G(a)[b]$ exists and satisfies
\begin{displaymath}
    D_{\mathcal A}^{+}\mathcal G(a)[b]
    =
    \left\langle \boldsymbol h(a),b\right\rangle_{L^2}
    =
    \sum_{i=1}^N
    \int_X h_i(x,a_i)\,b_i(x)\,d\mu(x).
\end{displaymath}
\end{definition}

\begin{remark}
\label{rem:abs_potential_using_dom}
    Definition \ref{def:potential_games}, the subsequent results, and their proofs can be adapted verbatim to the case where the total cost diverges. In this case, $\cA$ should be replaced by $\mathcal D_{\boldsymbol h}$, the domain of $\boldsymbol h$, which should be assumed to be convex.
\end{remark}

\subsubsection{A path-invariance characterization}
\label{subsubsec:path_characterization}

The advantage of the notion of SVNE$_\sigma$ when $\sigma=0$ is that, if
the game admits a potential as defined in Defition \ref{def:potential_games}, then \eqref{eq:svar_nash_eq_def} is the
first-order optimality condition of the optimisation problem:
$\min_{a\in\cA}\cG(a)$. We now determine exactly when such a $\cG$
exists. The characterization requires only the following radial continuity
property borrowed from the theory of monotone operators, see e.g.\
\cite[p.~291]{Galewski2021}:

\begin{assumption}[Radial continuity]\label{ass:radial_continuity}
    The operator $\boldsymbol h : \mathcal A \longrightarrow L^2(X,\mu;\mathbb R^N)$ is radially continuous is the following sense:
    
    \begin{equation}
    \forall \; a,a'\in\mathcal A, \quad t\mapsto
    \bigl\langle
        \boldsymbol h\bigl(a+t(a'-a)\bigr),\,a'-a
    \bigr\rangle_{L^2}
    \ \text{ is continuous on } [0,1].
\end{equation}
\end{assumption}

By convexity of $\mathcal A$ (Remark \ref{rem:first_properties_A}) the segment
$[a,a']$ is contained in $\mathcal A$, so Assumption \ref{ass:radial_continuity} is
well posed. 

Under Assumption \ref{ass:radial_continuity}, for any $a,a'\in\mathcal A$ the
line integral of $\boldsymbol h$ along the segment from $a$ to $a'$,
\begin{equation}\label{eq:line_integral}
    \mathcal I(a\to a')
    :=
    \int_0^1
    \bigl\langle
        \boldsymbol h\bigl(a+t(a'-a)\bigr),\,a'-a
    \bigr\rangle_{L^2}\,dt,
\end{equation}
is well defined since the integrand is continuous on $[0,1]$. The change of
variables $t\mapsto 1-t$ yields the orientation antisymmetry
\begin{equation}\label{eq:antisymmetry_I}
    \mathcal I(a'\to a) = -\,\mathcal I(a\to a'),
    \qquad \forall a,a'\in\mathcal A .
\end{equation}
Given points $a^0,\ldots,a^K\in\mathcal A$ with $a^K=a^0$, the
circulation of $\boldsymbol h$ along the closed polygonal path
$(a^0,\ldots,a^K)$ is $\sum_{k=0}^{K-1}\mathcal I(a^k\to a^{k+1})$.

\begin{theorem}[Path-invariance characterization of the potential]
\label{thm:path_characterization}
Let Assumption \ref{ass:radial_continuity} holds. The following statements are
equivalent.
\begin{enumerate}
    \item\label{item:pot} The game admits a potential in the sense of
    Definition~\ref{def:potential_games}.
    \item\label{item:tri} The circulation of $\boldsymbol h$ vanishes
    along every closed triangular path, i.e.\ for all
    $a^0,a^1,a^2\in\mathcal A$,
    \begin{equation}\label{eq:zero_circulation}
        \mathcal I(a^0\to a^1)
        +\mathcal I(a^1\to a^2)
        +\mathcal I(a^2\to a^0)
        =0 .
    \end{equation}
    \item\label{item:lineint} For every
    $\bar a\in\mathcal A$, the map
    \begin{equation}\label{eq:potential_formula}
        \mathcal G(a) := \mathcal I(\bar a\to a),
        \qquad a\in\mathcal A,
    \end{equation}
    is a potential.
\end{enumerate}
In this case the circulation of $\boldsymbol h$ vanishes along every
closed polygonal path contained in $\mathcal A$, any potential
$\widetilde{\mathcal G}$ satisfies
$\widetilde{\mathcal G}(a')-\widetilde{\mathcal G}(a)=\mathcal I(a\to a')$
for all $a,a'\in\mathcal A$, and the potential is unique up to an additive
constant.
\end{theorem}

\begin{proof}
The proof adapts the classical characterization of potential operators
by path-independence of the curvilinear integral
\cite[\S2.6 and \S6]{Vainberg1964} and is postponed to
Appendix~\ref{app:proof_thm_path}.
\end{proof}

\subsubsection{The differential symmetry test}
\label{subsubsec:symmetry_test}

Verifying condition \eqref{eq:zero_circulation} directly may be
cumbersome. Under smoothness, it reduces to a symmetry property of the
derivative, which is the test we shall actually use in examples. We first
recall the notion of differentiability relative to $\mathcal A$.

\begin{definition}[$C^1$ regularity relative to $\mathcal A$]
\label{def:C1-relative}
Let $E$ be a Banach space and let $F:\mathcal A\to E$. We say that $F$ is
of class $C^1$ \emph{relative to $\mathcal A$} if there exists a map
\begin{displaymath}
    D_{\mathcal A}F:\mathcal A\longrightarrow\mathcal L(\mathcal H_0,E),
\end{displaymath}
continuous when $\mathcal A$ is equipped with the topology induced by
$L^2(X,\mu;\mathbb R^N)$ and $\mathcal L(\mathcal H_0,E)$ with the
operator norm, such that for every $a\in\mathcal A$,
\begin{equation}\label{eq:C1-limit}
    \lim_{\substack{\widetilde a\to a \text{ in } L^2\\
                    \widetilde a\in\mathcal A,\;\widetilde a\neq a}}
    \frac{\bigl\|F(\widetilde a)-F(a)
          -D_{\mathcal A}F(a)\,[\widetilde a-a]\bigr\|_{E}}
         {\|\widetilde a-a\|_{L^2}}
    \;=\;0.
\end{equation}
By Remark~\ref{rem:structure_A}, $\widetilde a-a\in\mathcal H_0$ for all
$a,\widetilde a\in\mathcal A$, so the expression above is well defined.
\end{definition}

Note that
\eqref{eq:C1-limit} implies continuity of $F$ on $\mathcal A$ for the
strong topology, so that relative $C^1$ regularity of $\boldsymbol h$
implies the radial continuity given by Assumption \ref{ass:radial_continuity}. Two directions $b,b'\in\mathcal T_{\mathcal A}(a)$ are said to be
jointly feasible at $a$ if there exists $\varepsilon>0$ such that
$a+sb+tb'\in\mathcal A$ for every $s,t\in[0,\varepsilon]$.

\begin{proposition}[Symmetry characterization under relative $C^1$
regularity]
\label{prop:potential_condition}
Assume that the total cost induces an operator
$\boldsymbol h:\mathcal A\to L^2(X,\mu;\mathbb R^N)$ of class $C^1$
relative to $\mathcal A$. Then the game admits a potential if and only if
for every $a\in\mathcal A$ and every pair of jointly feasible directions
$b,b'\in\mathcal T_{\mathcal A}(a)$,
\begin{equation}
\label{eq:symmetry_condition}
    \left\langle
        D_{\mathcal A}\boldsymbol h(a)[b],b'
    \right\rangle_{L^2}
    =
    \left\langle
        D_{\mathcal A}\boldsymbol h(a)[b'],b
    \right\rangle_{L^2}.
\end{equation}
\end{proposition}

\begin{proof}
The proof adapts Vainberg's symmetry criterion for potential operators
\cite[Thm~5.1]{Vainberg1964} to the convex set $\mathcal A$, and is
postponed to Appendix~\ref{app:proof_prop_symmetry}.
\end{proof}

\paragraph{Examples of costs satisfying the symmetry condition.}
A typical example for condition \eqref{eq:symmetry_condition} to hold is
\begin{displaymath}
    h_i(x,a_i)= c_i(x)+\int_X K_i(x,y)a_i(y)\,d\mu(y),\quad\forall i\in I,
\end{displaymath}
with $K_i\in L^2(X\times X,\mu\otimes\mu)$ symmetric and $c_i\in
L^2(X,\mu)$: the derivative
$D_{\mathcal A}\boldsymbol h(a)[b]=\bigl(\int_X
K_i(\cdot,y)b_i(y)d\mu(y)\bigr)_i$ is constant in $a$ and symmetric, so
that, by Proposition~\ref{prop:potential_condition}, the game admits a potential. The potential structure of such games yields the convergence of
the $N$-player Nash equilibria to the equilibrium of the continuum game
\cite{machado2025nash}.

\medskip
\noindent
A particular case is $K_i\equiv 1$, in which the interaction occurs only
through the mass. More generally, suppose that
\begin{displaymath}
    h_i(x,a_i) = c_i(x) + \tilde h_i\left(\int_X a_i\,d\mu\right),
\end{displaymath}
with $\tilde h_i\in C^1(\R)$. Then $\boldsymbol h$ is of class $C^1$
relative to $\mathcal A$, with
\begin{displaymath}
    D_{\mathcal A}\boldsymbol h(a)[b]
    =\left(
        \tilde h_i'\!\left(\int_X a_i\,d\mu\right)
        \int_X b_i\,d\mu
    \right)_{i\in I},
\end{displaymath}
so that, for all directions $b,b'\in\mathcal H_0$,
\begin{displaymath}
    \bigl\langle D_{\mathcal A}\boldsymbol h(a)[b],\,b'\bigr\rangle_{L^2}
    =\sum_{i\in I}
    \tilde h_i'\!\left(\int_X a_i\,d\mu\right)
    \left(\int_X b_i\,d\mu\right)
    \left(\int_X b_i'\,d\mu\right)
    =\bigl\langle D_{\mathcal A}\boldsymbol h(a)[b'],\,b\bigr\rangle_{L^2}:
\end{displaymath}
the symmetry condition \eqref{eq:symmetry_condition} holds, and
Proposition~\ref{prop:potential_condition} shows that the game admits a
potential. It is the classical potential of Beckmann, McGuire and
Winsten \cite{beckmann1956studies}: up to an additive constant,
\begin{displaymath}
    \mathcal G(a)
    =\sum_{i\in I}\int_X c_i(x)a_i(x)\,d\mu(x)
    +\sum_{i\in I}\tilde H_i\!\left(\int_X a_i\,d\mu\right),
    \qquad \tilde H_i':=\tilde h_i,
\end{displaymath}
which is the setting of \cite{crippa2009optimum}. Note that when the
$\tilde h_i$ are merely continuous, the same conclusion holds by
checking directly that $\mathcal G$ satisfies
Definition~\ref{def:potential_games}, the one-dimensional chain rule
requiring no more than continuity of $\tilde H_i'=\tilde h_i$.

\medskip
However, when $h_i$ depends on other aggregate quantities, meaning that
players contribute unequally to congestion, the potential structure is
destroyed.

\begin{proposition}[Heterogeneous contributions preclude a potential]
\label{prop:aggregate}
Let $N\ge2$ and let $\psi^1,\ldots,\psi^m\in L^2(X,\mu;\R_+)$ be such
that the constant function $\mathbf 1$ does not belong to $
\operatorname{span}\{\psi^1,\ldots,\psi^m\}$. Suppose that, for every $i\in I$,
\begin{displaymath}
    h_i(x,a_i)
    =
    c_i(x)
    +
    \tilde h_i\!\left(\Psi a_i\right),
\end{displaymath}
with $c_i\in L^2(X,\mu)$ and $\tilde h_i\in C^1(\R^m)$ coordinate-wise
non-decreasing. If, for some $i_0\in I$ and some
$\ell_0\in\{1,\ldots,m\}$, the map $\tilde h_{i_0}$ is strictly increasing
in its $\ell_0$-th variable and $\psi^{\ell_0}$ is not $\mu$-a.e.\ zero, and there exists $\ell_1$ such that $\psi^{\ell_0}\notin\operatorname{span}\psi^{\ell_1}$,
then the game does not admit a potential.
\end{proposition}

\begin{proof}
Without loss of generality, suppose $i_0=1,\;\ell_0 = 1$ and $\ell_1=2$. For any $i$, we write $m_i(a):=\Psi a_i\in\R^m$. Direct computations show that
$\boldsymbol h$ is of class $C^1$ relative to $\mathcal A$ with
$$D_{\mathcal A}\boldsymbol h(a)[b]
=\bigl(\sum_{\ell}D_{\ell}\tilde h_i(m_i(a))
\int_X\psi^\ell b_i\,d\mu\bigr)_{i\in I},$$ 
where $D_\ell \tilde{h}_i$ denotes the partial derivative of $\tilde{h}_i$ w.r.t. its $\ell^{th}$ argument. Previous equality yields for all
$b,b'\in\mathcal T_{\mathcal A}(a)$,
\begin{equation}\label{eq:antisym}
    \bigl\langle D_{\mathcal A}\boldsymbol h(a)[b],b'\bigr\rangle_{L^2}
    -
    \bigl\langle D_{\mathcal A}\boldsymbol h(a)[b'],b\bigr\rangle_{L^2}
    =
    \sum_{k\in I}\sum_{\ell=1}^m
    D_{\ell}\tilde h_k\bigl(m_k(a)\bigr)
    \Bigl[
    \textstyle\int\psi^\ell b_kd\mu\int b_k'd\mu
    -\int\psi^\ell b_k'd\mu\int b_kd\mu
    \Bigr].
\end{equation}
Let us fix two directions. First, $\psi^1$ is not
$\mu$-a.e.zero and does not belong to $ \operatorname{span}\{\psi^{2}\}$, so there exists $u\in L^\infty(X,\mu)$ such that 
\begin{displaymath}
    u\ge 0,\quad \int_X\psi^{1}u\,d\mu>0,\quad \int_X\psi^{2}u\,d\mu = 0.
\end{displaymath}
Since the
$\psi^\ell$ are nonnegative, then $\int_X\psi^\ell u\,d\mu$ is nonnegative for every $\ell$. Namely
\begin{displaymath}
    \alpha:=\left(\int_X\psi^1 u\,d\mu,\ldots,\int_X\psi^m u\,d\mu\right)\; \in\; [0,+\infty)^m. 
\end{displaymath} Second, since $\mathbf 1\notin
\operatorname{span}\{\psi^1,\ldots,\psi^m\}$, there exists
$v\in L^\infty(X,\mu)$ such that
\begin{displaymath}
    \int_X\psi^\ell v\,d\mu=0\quad\forall\ell,
    \qquad
    \beta:=\int_X v\,d\mu>0 .
\end{displaymath}
Rescale so that $\|u\|_\infty,\|v\|_\infty\le1/(8N)$, and define
\begin{displaymath}
    \bar a:=(1/N,\ldots,1/N),\quad b:=(u,-u,0,\ldots,0),\quad b':=(v,-v,0,\ldots,0).
\end{displaymath}
The map
\begin{displaymath}
    [-1,1]\ni\theta\mapsto\tilde h_{i_0}\bigl(m_1(\bar a)+\theta m_1(b)\bigr)
\end{displaymath} is increasing and of class $C^1$, then there exists a positive constant $\theta^*$ such that its derivative
\begin{displaymath}
    D_{\psi^1}\tilde h_{1}\bigl(m_1(\bar a)+\theta^* m_1(b)\bigr)\alpha_1>0.
\end{displaymath}
Therefore, set $a=\bar a+\theta^* b$. The smallness
of $u,v$ ensures that the components of $a+sb+tb'$, $s,t\in[0,1]$,
remain in $(0,1)$ and sum to $1$. In particular,  $a\in\mathcal A$ and $b,b'$ are jointly feasible
at $a$.

\smallskip
\noindent
We evaluate \eqref{eq:antisym} at this $a$. Only $k=1,2$ contribute. In
each bracket, the second term vanishes since
$\int\psi^\ell b_k'd\mu=0$ for every $\ell$. In the first
term, $\int\psi^\ell b_1d\mu=\alpha_\ell=-\int\psi^\ell b_2d\mu$ and
$\int b_1'd\mu=\beta=-\int b_2'd\mu$, so the two signs cancel and
\eqref{eq:antisym} reduces to
\begin{displaymath}
    \beta\sum_{\ell=1}^m
    \Bigl[
    D_{\ell}\tilde h_1\bigl(m_1(a)\bigr)
    +D_{\ell}\tilde h_2\bigl(m_2(a)\bigr)
    \Bigr]\alpha_\ell
    \;\ge\;
    \beta
    D_{\psi_{1}}\tilde h_1\bigl(m_1(\bar a) + \theta^* m_1(b)\bigr)\alpha_{1}
    \;>\;0.
\end{displaymath}
The first inequality comes from using $D_{\ell}\tilde h_i\ge0$ and $\alpha_\ell\ge0$,
and the last strict inequality being the choice of $\theta^*$. By
Proposition~\ref{prop:potential_condition}, the game admits no
potential.
\end{proof}

\begin{remark}
The $C^1$ regularity of the maps $\tilde h_i$ is only used to invoke
Proposition~\ref{prop:potential_condition}. When the $\tilde h_i$ are
merely continuous, the same conclusion can be derived from
Theorem~\ref{thm:path_characterization}, by exhibiting a closed path
along which the circulation of $\boldsymbol h$ does not vanish. We have
preferred the differential argument for its brevity.
\end{remark}
\begin{remark}[Comparison with weighted congestion games]
\label{rem:literature_no_potential}
In atomic congestion games, heterogeneous weights break the potential
structure: a potential exists only for
affine or exponential costs \cite{harks2011characterizing}, and outside
these classes a pure Nash equilibrium may fail to exist
\cite{FKS2005}. 
Similarly, Theorem~\ref{thm:path_characterization} and Proposition~\ref{prop:potential_condition} characterise the conditions under which a potential exists and Proposition~\ref{prop:aggregate} provides counterexamples when these conditions are not satisfied.
\end{remark}

\begin{remark}[Weighted potentials and the role of multiple aggregates]
\label{rem:weighted_rescue}
Proposition~\ref{prop:aggregate} rules out a potential. When
$m=1$ and $\psi>0$ $\mu$-a.e., a \emph{weighted} potential nevertheless
always exists: rescaling each player's costs by $\psi(x)$ leaves best responses
unchanged and yields a potential game, so that equilibria can still be
obtained by minimizing a (convex, under
Assumption~\ref{ass:uniqueness_interaction_term}) functional. No such rescaling exists in general for
$m\ge2$ independent aggregates, as in the application of
Section~\ref{illustrative_ex}. 
\end{remark}

In view of Proposition~\ref{prop:aggregate}, when players' contributions to
congestion differ, the potential structure of the game fails, and the
equilibrium can no longer be obtained by optimisation. The rest of the
paper is devoted to establishing the existence and uniqueness of an
equilibrium without relying on a potential.

\subsection{First properties of the SVNE\texorpdfstring{$_\sigma$}{Lg}}\label{subsec:prop_svne}

This section establishes three key properties of the SVNE$\sigma$. First, Proposition \ref{prop:justification_SVNE} confirms that the concept of SVNE$_\sigma$ is related to the concept of Nash equilibrium for any non-negative $\sigma$. Specifically, it shows that at equilibrium, any resource chosen by a player is necessarily a cost-minimizing one. Second, Proposition \ref{prop:voronoi_type_result} demonstrates that any SVNE$_\sigma$ is a pure-strategy equilibrium. Finally, for the case where $\sigma > 0$, Corollary \ref{cor:SVNE_fixed_point} makes explicit the connection between an SVNE$_\sigma$ and a QRE.

We emphasize the distinct roles of our main assumptions. The absolute continuity of the player distribution $\mu$ is instrumental for all results in this section. In contrast, the no-tie condition (Assumption \ref{ass:transport_costs}) is required only for the specific case $\sigma=0$, and is used in the proofs of Lemma \ref{lem:intersection_mesure_zero} and Proposition \ref{prop:voronoi_type_result}.
    \begin{proposition}\label{prop:justification_SVNE}
    Let $\sigma\geq 0$, under Assumptions \ref{ass:transport_costs} and \ref{ass:existence_interaction_term}, the two following statements are equivalent:
        \begin{enumerate}
            \item $f$ is a SVNE$_\sigma$.
            \item For any index $i\in I$, if $(\varepsilon, x)$ is a Lebesgue point of $f$ and $\nu^N\otimes \mu$, and that $f_i(\varepsilon,x) \nu^N(\varepsilon)\mu(x)>0$, then  
            \begin{displaymath}
                c_i(x)+h_i(a_i)-\sigma \varepsilon_i=\min_{k\in I} c_k(x)+h_k( a_k)-\sigma \varepsilon_k,
            \end{displaymath}
        \end{enumerate}
         where $a_k(x) := \int_{\R^N} f_k(\varepsilon,x)d\nu^N(\varepsilon)$ for any $(x,k)\in X\times I$.
    \end{proposition}
    \begin{proof}
        \textit{Step 1.} Suppose $f$ is a SVNE$_\sigma$.
        Fix the indices $i$ and $j\neq i$. We fix $(\varepsilon^0,x^0)$ such that $B_r(\varepsilon^0,x^0)$ is included in $\mathbb{R}^N\times X$, it is a Lebesgue point of $f$ and $\nu^N\otimes \mu$, and that the product $f_i(\varepsilon,x)\nu^N(\varepsilon)\mu(x)$ is positive.
        Let us define $g$ as follows: 
        \begin{itemize}
            \item for all $k\neq i,j$, set $g_k := f_k$
            \item set $g_j := f_j + f_i\ind_{B_r(\varepsilon^0,x^0)}$,
            \item and set $g_i := f_i(1-\ind_{B_r(\varepsilon^0,x^0)})$. 
        \end{itemize}
        One can directly check that $g$ is in $\cF$.
        The definition of a SVNE$_\sigma$ leads to
        \begin{align*}
            &\int_X\int_{\R^N}(c_i(x)+h_i(a_i)-\sigma\varepsilon_i)f_i(\varepsilon,x)\ind_{B_r(\varepsilon^0,x^0)} d\nu^N(\varepsilon)d\mu(x)\\
            &\quad\le \int_X\int_{\R^N}(c_j(x)+h_j(a_j)-\sigma\varepsilon_j)f_i(\varepsilon,x)\ind_{B_r(\varepsilon^0,x^0)}d\nu^N(\varepsilon)d\mu(x).
        \end{align*}
        Dividing by $\abs{B_r(\varepsilon^0,x^0)}$ and letting $r$ goes to zero, we invoke the continuity of $c$ (Assumption \ref{ass:transport_costs}) and the Lebesgue differentiation theorem. Then dividing by $f_i(\varepsilon^0,x^0)\mu(x^0)\nu^N(\varepsilon^0)$, we obtain that:
        \begin{displaymath}
            c_i(x^0) + h_i(a_i) - \sigma \varepsilon^0_i \le c_j(x^0) + h_j(a_j) - \sigma\varepsilon^0_j.
        \end{displaymath}
        \bigskip
        \noindent
        \textit{Step 2.}
        Conversly, almost all pair $(\varepsilon,x)$ is a Lebesgue point of $f$ and $\nu^N\otimes \mu$.
        Therefore, 
        \begin{align*}
        &\sum_{i\in I}\int_X\int_{\R^N} (c_i(x)+h_i(a_i)-\sigma\varepsilon_i)f_i(\varepsilon,x)d\nu^N(\varepsilon)d\mu(x)\\
        &\quad= \int_X\int_{\R^N} \min_{j\in I}(c_j(x)+h_j(a_j)-\sigma \varepsilon_j)d\nu^N(\varepsilon)d\mu(x),
    \end{align*}
    leading to \eqref{eq:svar_nash_eq_def}.
    \end{proof}
    In the notion of SVNE$_\sigma$, players are allowed to play in mixed strategies. Lemma \ref{lem:intersection_mesure_zero} below prepares Proposition \ref{prop:voronoi_type_result} ensuring that under rather general assumptions players do not play with mixed strategies.
    \begin{lemma}\label{lem:intersection_mesure_zero}Let $\sigma\geq 0$ and $T\in\R^N$, we introduce for any $i\in \{1,\ldots,N\}$ the set:
    \begin{displaymath}
        A_i = \{(\varepsilon,x)\in\R^N\times X\;:\;c_i(x) + T_i-\sigma\varepsilon_i = \min_{k\in I}c_k(x) + T_k-\sigma\varepsilon_k\}.
    \end{displaymath}
    Under Assumption \ref{ass:transport_costs}, if $i\neq j$, then  $\nu^N\otimes \mu(A_i\cap A_j) = 0$. 
        
    \end{lemma}

    \begin{proof}
        By contradiction, let us assume that $\mu(A_i\cap A_j)$ is positive. There are two cases: 
        \begin{enumerate}
            \item When $\sigma>0$. Let $E_{x} = \{\varepsilon \in \R^N \mid \varepsilon_j - \varepsilon_i = \frac{1}{\sigma}((c_j(x)+T_j) - (c_i(x)+T_i))\}$. Since the law of $(\mathcal{E}_1, \dots, \mathcal{E}_N)$ under $\nu^N$ is absolutely continuous with respect to the Lebesgue measure on $\R^N$, the set $E_x$, which is a hyperplane, has measure zero: $\nu^N(E_x) = 0$.
    
    The set $A_i\cap A_j$ is included in the set $\{(\varepsilon, x) \mid \varepsilon_j - \varepsilon_i = \frac{1}{\sigma}((c_j(x)+T_j) - (c_i(x)+T_i))\}$. By Fubini's theorem, the measure of this latter set is
    \begin{displaymath}
        \int_X \nu^N(E_x) d\mu(x) = \int_X 0 \, d\mu(x) = 0.
    \end{displaymath}
    Since $A_i\cap A_j$ is a subset of a set of measure zero, we must have $\nu^N\otimes \mu(A_i\cap A_j)=0$. This contradicts our initial assumption.

    \item When $\sigma = 0$, the equality becomes:
    \begin{displaymath}
        c_i(x)+T_i = c_j(x)+T_j.
    \end{displaymath}
    The random shocks $\varepsilon$ have disappeared and the equality only depends on $x$. We use the no-tie hypothesis of Assumption \ref{ass:transport_costs} which states that for any distinct pair $i,j$, the set of ties is negligible. That is, the set
    \begin{displaymath}
        X_{ij}^{T} = \{x \in X \mid c_i(x)+T_i = c_j(x)+T_j\},
    \end{displaymath}
    has measure zero with respect to $\mu$, i.e., $\mu(X_{ij}^{T}) = 0$.
    
    The set $A_i\cap A_j$ from our initial hypothesis must be a subset of $\R^N\times X_{ij}^{T}$. The measure of this set is
    \begin{displaymath}
        \nu^N\otimes \mu(\R^N \times X_{ij}^{T} ) = \nu^N(\R^N)\times \mu(X_{ij}^{T}) = 0.
    \end{displaymath}
    Therefore, $A_i\cap A_j$ is a subset of a set of measure zero, which implies $\nu^N\otimes \mu(A_i\cap A_j)=0$. This is again a contradiction.
    \end{enumerate}
    \end{proof}
    \begin{proposition}\label{prop:voronoi_type_result}Let $\sigma\geq 0$, under Assumption \ref{ass:transport_costs} and \ref{ass:existence_interaction_term}, $f$ is a SVNE$_\sigma$ if and only if $\nu^N\otimes \mu$-a.s.
        \begin{equation}
            \label{eq:f_formula}
            f_i(\varepsilon,x)=
                \begin{cases}
                1, & i \in \arg\min_{j\in I}\left\{c_j(x)+h_j(a_j)-\sigma \varepsilon_j\right\},\\
                0, & \text{otherwise},
                \end{cases}
            \end{equation}
         where $a_k(x) := \int_\R f_k(\varepsilon,x)d\nu^N(\varepsilon)$ for any $(x,k)\in X\times I$.
    \end{proposition}
    \begin{proof}
    
    $ $\newline
    \noindent
    \textit{Step 1. The direct implication.} Let $f$ be a SVNE$_\sigma$. Let us use the sets $A_i$ introduced in Lemma \ref{lem:intersection_mesure_zero} with $T_k = h_k(a_k)$ for any $k\in I$, where $a_k := \int_{\R^N}f_k(\varepsilon,x)d\nu^N(x)$.

    \bigskip
         
         For any Lebesgue point $(\varepsilon,x)$ of $f$ and $\nu^N\otimes \mu$, if $f_i(\varepsilon,x)>0$, $f_j(\varepsilon,x)>0$, and $\nu^N(\varepsilon)\mu(x)>0$, then, from Proposition \ref{prop:justification_SVNE}, $(\varepsilon,x)$ belongs to $A_i\cap A_j$.
        Lemma \ref{lem:intersection_mesure_zero} ensures that the set formed by such pair $(\varepsilon,x)$ is $\nu^N\otimes \mu$-negligible. 

        \smallskip
        \noindent
        We then deduce that for a fix $i\in I$, for $\nu^N\otimes \mu$-a.e. $(\varepsilon,x)$ in $\mathbb{R}^N\times X$, if $f_i(\varepsilon,x)>0$, then for any $j\neq i$, $f_j(\varepsilon,x)=0$ and so $f_i(\varepsilon,x)= 1$.

    \bigskip
    \noindent
    Let $(\varepsilon,x)$ be a Lebesgue point of $f$ and $\nu^N\otimes \mu$, suppose that $f_i(\varepsilon, x) = 0$, then there exists $j$ such that $f_j(\varepsilon,x)>0$ and from Proposition \ref{prop:justification_SVNE}
    \begin{displaymath}
        (\varepsilon,x)\in A_j.
    \end{displaymath}
    Therefore, for a.e. Lebesgue point $(\varepsilon,x)\in \mathbb{R}^N\times X$ of $f$ and $\nu^N\otimes\mu$, 
    \begin{enumerate}
        \item either $f_i(\varepsilon,x) >0 $, and so $f_i(\varepsilon,x)=1$ and $(\varepsilon ,x)\in A_i$, 
        \item or $f_i(\varepsilon,x) = 0$, and so there exists $j\neq i$ such that $(\varepsilon,x)\in A_j$. 
    \end{enumerate}
    The second point is equivalent to:  for $\nu^N\otimes \mu$-a.e. $(\varepsilon,x)$ in $\mathbb{R}^N\times X$, if $(\varepsilon,x)\in X\backslash \cup_{j\neq i} A_j$, then  $f_i(\varepsilon,x)>0$ implying $f_i(\varepsilon,x)=1$. \newline
    By Lemma \ref{lem:intersection_mesure_zero} for any $j\neq i$, $A_i\cap A_j$ is negligible, then almost surely $f_i=\ind_{A_i}$. Since it is true for any $i$, we conclude the proof. 

    \bigskip
    \noindent 
    \textit{Step 2. The converse implication.} Suppose $f$ has the form \eqref{eq:f_formula} with $a\in\cA$. We observe that
    \begin{displaymath}
        \sum_i \int_X\int_{\mathbb{R}^N}(c_i(x)+h_i(a_i)-\sigma\varepsilon_i)f_i(\varepsilon,x) d\nu^N(\varepsilon)d\mu(x) =   \int_X\int_{\R^N}\min_j(c_j(x)+h_j(a_j)-\sigma\varepsilon_j)d\nu^N(\varepsilon)d\mu(x),
    \end{displaymath}
    from which \eqref{eq:svar_nash_eq_def} is deduced. 
    \end{proof}
Two corollaries are derived from Proposition \ref{prop:voronoi_type_result}. Corollary \ref{cor:equivalence} shows the equivalence between a PNE and a SVNE$_0$. Then, Corollary \ref{cor:SVNE_fixed_point} states that for $\sigma>0$, finding a SVNE$_\sigma$ boils down to solving a fixed-point equation.

\begin{corollary}\label{cor:equivalence}
Under Assumption \ref{ass:transport_costs} and \ref{ass:existence_interaction_term}, the two following statements hold.
\begin{enumerate}
    \item Let $(A_1,\ldots,A_N)$ be PNE, and $g\in \mathcal{F}$ be defined by
    \begin{equation*}
        g_i(\varepsilon,x):=\ind_{A_i}(x),
    \end{equation*}
    then $g$ is a SVNE$_0$.
    \item Let $f\in \mathcal{F}$ be a SVNE$_0$, let $a$ be defined from $f$ as in Proposition \ref{prop:voronoi_type_result} and let $(B_1,\ldots,B_N)$ be defined by
    \begin{equation*}
        B_i:=\displaystyle \left\{ x\in X \,\big\vert\,a_i(x)=1 \right\}.
    \end{equation*}
    Then $(B_1,\ldots,B_N)$ is, up to a $\mu$-null, a PNE.
\end{enumerate}
\end{corollary}
\begin{proof}
\begin{enumerate}
    \item Since $(A_1,\ldots,A_N)$, by its definition $g$ satisfies $\nu^N\otimes \mu$-a.s.
        \begin{equation*}
            g_i(\varepsilon,x)=
                \begin{cases}
                1, & c_i(x)+h_i(g_i) =\min_{j\in I}\left\{c_j(x)+h_j(g_j)\right\},\\
                0, & \text{otherwise}.
                \end{cases}
            \end{equation*}
Then, by Proposition \ref{prop:voronoi_type_result} $g$ is a SVNE$_0$.

\item By Proposition \ref{prop:voronoi_type_result}, for any $i\in \{1,\ldots,N\}$ the set $B_i$ satisfies $\nu^N\otimes \mu$-a.s
\begin{equation*}
    B_i=\displaystyle \left\{ x\in X \,\big\vert\,c_i(x)+h_i(a_i)=\min_{j\in I}\left\{c_j(x)+h_j(a_j) \right\}\right\},
\end{equation*}
where $a$ is defined from $f$ as in Proposition \ref{prop:voronoi_type_result}. Then, by Lemma \ref{lem:intersection_mesure_zero}, $(B_1,\ldots,B_N)$ is a partition of $X$.
Further, from Proposition \ref{prop:voronoi_type_result} is holds for $\mu$-a.e. $x$ in $X$
\begin{equation*}
    a_i(x)=
        \begin{cases}
        1, & f_i(\varepsilon,x)=1 \quad \nu^N-a.s.,\\
        0, & \text{otherwise}.
        \end{cases}
    \end{equation*}
The two previous equalities and equation \eqref{eq:f_formula} yield
for $\mu$-a.e. $x$ in $X$
\begin{equation*}
    a_i(x)=\ind_{B_i}(x).
\end{equation*}
Then, $(B_1,\ldots,B_N)$ is a PNE.
\end{enumerate}
\end{proof}

\begin{corollary}
    \label{cor:SVNE_fixed_point}
    Under Assumption \ref{ass:transport_costs} and \ref{ass:existence_interaction_term}, if $\sigma>0$, the two following statements are equivalent:
    \begin{enumerate}
        \item $f$ is a SVNE$_\sigma$. 
        \item $f$ is defined $\nu^N\otimes\mu$-a.s.            \begin{equation}\label{eq:formula_f}
            f_i(\varepsilon,x)=
            \begin{cases}
            1, & i \in \arg\min_{j}\left\{c_j(x)+h_j(a_j)-\sigma \varepsilon_j\right\},\\
            0, & \text{otherwise}.
            \end{cases}
        \end{equation}
        where $a\in \mathcal{A}$ is a fixed point solution of 
        \begin{equation*}
            a_i(x) = \bP\big(c_i(x)+h_i(a_i)-\sigma\mathcal{E}_i = \min_{j\in I}(c_j(x)+h_j(a_j)-\sigma\mathcal{E}_j\big),
        \end{equation*}
        where the probability $\bP$ is defined in Definition \ref{def:SVNE}.
    \end{enumerate}

\end{corollary}
\begin{proof}Let us prove that 1 implies 2. From Proposition \ref{prop:voronoi_type_result}, $f$ is a SVNE$_\sigma$ implies it has the form \eqref{eq:formula_f}. Also $a$ must satisfy for any $i$, $a_i(x) = \E[f_i(\mathcal{E},x)]$. Therefore,
    \begin{align*}
        a_i(x) &=\E[f_i(\mathcal{E},x)]\\
        &=  \bP(f_i(\mathcal{E},x)=1)\\
        &=\bP(c_i(x)+h_i(a_i)-\sigma\mathcal{E}_i = \min_{j\in I}(c_j(x)+h_j(a_j)-\sigma\mathcal{E}_j)
    \end{align*}
    Conversely, suppose that $a$ is a fixed-point and $f$ is given by \eqref{eq:formula_f}. From the definition of $f$ it is easy to establish \eqref{eq:svar_nash_eq_def}. It remains to check that $a_i(x)$ is the expectation of $f_i(\mathcal{E},x)$. Starting from $\E[f_i(\mathcal{E},x)]$ we use the same equalities as before to deduce the result.  
    
\end{proof}

Table~\ref{tab:equilibrium_concepts} summarizes the relationships between the
equilibrium concepts considered in the paper. In particular, it highlights
the distinction between the noise-contingent equilibrium strategy described
by the SVNE and the choice probabilities retained by the QRE.

\begin{center}
\setlength{\belowcaptionskip}{7pt}
\captionof{table}{Equilibrium concepts considered in the paper, under the assumptions
of Proposition~\ref{prop:voronoi_type_result}. \textit{For a QRE, the noise-contingent strategy is not part of the equilibrium
description, which directly specifies the choice probabilities conditional on
\(x\). However, its equivalent SVNE$_\sigma$ representation shows that these
probabilities can be generated by choices that are pure conditional on
\((\varepsilon,x)\)}.}
\label{tab:equilibrium_concepts}
\renewcommand{\arraystretch}{1.4}
\setlength{\extrarowheight}{2pt}
\setlength{\tabcolsep}{3pt}
\small
\begin{tabular}{
|C{2.1cm}
|C{1.3cm}
|C{2.5cm}
|C{3.5cm}
|C{3.7cm}|
}
\hline
\celltop
\shortstack{\textbf{Equilibrium}\\\textbf{concept}}
&
\textbf{Noise}
&
\shortstack{\textbf{Admissible}\\\textbf{strategies}}
&
\shortstack{\textbf{Equilibrium choice}\\\textbf{given all information}}
&
\shortstack{\textbf{Equilibrium allocation}\\
\textbf{given \(x\)}}
\\
\hline

PNE
&
None
&
Pure
&
Pure given \(x\)
&
Pure given \(x\)
\\
\hline

SVNE$_0$
&
None
&
Mixed
&
Pure given \(x\)
&
Pure given \(x\)
\\
\hline
\shortstack{SVNE$_\sigma$, \(\sigma>0\)}
&
Yes
&
Mixed
&
Pure given \((\varepsilon,x)\)
&
Probabilistic given \(x\)
\\
\hline

QRE
&
Yes
&
Not represented
&
Not represented
&
Probabilistic given \(x\)
\\
\hline

\end{tabular}
\end{center}

Note that in a finite players game setting, the fixed-point equation 
\begin{equation*}
a_i(x)=    \bP(c_i(x)+h_i(a_i)-\sigma\mathcal{E}_i = \min_{j\in I}(c_j(x)+h_j(a_j)-\sigma\mathcal{E}_j) 
\end{equation*}
is usually the definition of a Quantal Response Equilibrium \cite{mckelvey1995quantal}. Therefore, in the non-atomic framework, we have proved that the notion of SVNE$_\sigma$ with $\sigma>0$ is equivalent to the QRE if the distribution of the player $\mu$ is absolutely continuous. 
A particular important case in  \cite{mckelvey1995quantal} is the \textit{Logit Equilibrium}, i.e. when $(\mathcal{E}_1,\ldots,\mathcal{E}_N)$ are i.i.d. and follow a standard Gumbel law (Assumption \ref{density_assum}). In this setting, the probability for a player to choose $i$ can be explicitly expressed, i.e.
\begin{equation}\label{def_svne_gumbel}
    \bP\left(c_i(x)+h_i(a_i)-\sigma\mathcal{E}_i = \min_{j\in I}(c_j(x)+h_j(a_j)-\sigma\mathcal{E}_j\right) = \frac{e^{-\frac{c_i(x) + h_i(a_i)}{\sigma}}}{\displaystyle\sum_{j\in I}e^{-\frac{c_j(x) + h_j(a_j)}{\sigma}}}.
\end{equation}
A well-known property of the \textit{softmin operator} is that for any $T$ in $\mathbb{R}^N$, $i$ in $I$ and $\mu$-a.e. $x$ in $X$ it holds
    \begin{equation}
        \label{pointwise_conv}
        \lim_{\sigma\to 0} \frac{e^{-\frac{c_i(x) + T_i}{\sigma}}}{\displaystyle\sum_{j\in I}e^{-\frac{c_j(x) + T_j}{\sigma}}}=
\left\{
\begin{array}{ll}
  1   & \mbox{if }c_i(x) + T_i=\displaystyle\min_{j\in I}c_j(x) + T_j, \\
  0   &\mbox{otherwise}.
\end{array}
     \right.
    \end{equation}

\subsection{Existence of equilibria}\label{equilibria_section}

The main results of this section are the existence of SVNE$_\sigma$ for any positive $\sigma$ (Theorem \ref{existence_SUE}) and of a SVNE$_0$ and a PNE (Theorem \ref{thm:existence_PNE}).\newline
 First, for any parameter $\sigma > 0$, we prove in Section \ref{eq_under_noise} the existence of a SVNE$_\sigma$ when $\mathcal{E}_1,\ldots,\mathcal{E}_N$ are i.i.d and follow a standard Gumbel law (Assumption \ref{density_assum}). To do so, we apply the fixed-point characterization given in Corollary \ref{cor:SVNE_fixed_point} in the previous section. Second, in Section \ref{no_noise_eq_section}, taking the limit as $\sigma \to 0$, we establish the existence of a SVNE$_0$ through a limiting argument. Finally, using Proposition \ref{prop:voronoi_type_result}, we conclude that from this resulting SVNE$_0$ we can build a PNE. 

\subsubsection{Existence of equilibria under noise}\label{eq_under_noise}

Fix $\sigma>0$, and let $F:\cA\to\cA $ be defined for any $x\in X$ and $i\in I$ by
\begin{equation}\label{def_F}
    F_i(a)(x):= \frac{e^{-\frac{c_i(x) + h_i(a_i)}{\sigma}}}{\displaystyle\sum_{j\in I}e^{-\frac{c_j(x) + h_j(a_j)}{\sigma}}}.
\end{equation}
In view of Corollary \ref{cor:SVNE_fixed_point} and equation \eqref{def_svne_gumbel}, finding a SVNE$_\sigma$ is equivalent to solve:
\begin{equation}\label{fixed_point_F}
       F(a) = a.  
\end{equation}
The lemmas below make sure that the Leray-Shauder fixed-point theorem can be used. Specifically, Lemma \ref{upbound_min_h} gives a uniform upper bounds for the minimum of $h$ on $\cA$. From this bound, Lemma \ref{F_well_defined} ensures that $F:\cA\rightarrow \cA$ is well-defined. Lemma \ref{lem:cont_F} establishes the continuity of $F$. All the ingredients are then reunited to establish the existence of a fixed-point.

\begin{lemma}\label{upbound_min_h}
    Under Assumption \ref{ass:transport_costs} and \ref{ass:existence_interaction_term}, there exists $\tau>0$ such that for any $a\in \cA$, it holds 
    \begin{equation*}
      \min_{i\in I}h_i( a_i)\leq\tau.
    \end{equation*}
\end{lemma}

\begin{proof}
The result is obtained by contradiction. Assume there exists a sequence $(a^n)_{n\in\mathbb{N}}$ in $\mathcal{A}$ such that
\begin{equation}
    \label{inf_h_seq}
    \lim_{n\to \infty}\min_{i\in I }h_i(a^n_i) = +\infty.
\end{equation}
Since $L^2(X,\mu;\mathbb{R}^N)$ is reflexive and $\mathcal{A}$ is bounded, convex, and strongly closed, $\mathcal{A}$ is weakly compact in $L^2(X,\mu;\mathbb{R}^N)$.
Since the set $\mathcal{A}$ is weakly compact in $L^2(X,\mu;\mathbb{R}^N)$, there exists a sub-sequence of $(a^n)_{n\mathbb{N}}$, denoted in the same way by ease of notations, weakly-converging in $L^2(X,\mathbb{R}^N)$ to $a^\infty\in \cA$. By Assumption \ref{ass:existence_interaction_term} there exists $p \in I$ such that 
\begin{equation*}
    h_p(a^\infty_p)<\infty.
\end{equation*}
By Assumption \ref{ass:existence_interaction_term}, the domain of $h_p$ is an open
and using the weak continuity of $h_p$ over its domain, it holds that
\begin{equation*}
    \lim_{n\to \infty}h_p( a_p^{n}) = h_p(a^\infty_p) < +\infty,
\end{equation*}
which contradicts \eqref{inf_h_seq}.
\end{proof}

\begin{lemma}\label{F_well_defined}
    Under Assumption \ref{ass:transport_costs} and \ref{ass:existence_interaction_term}, the map $F$ is well defined on $\mathcal{A}$ and maps $\mathcal{A}$ to itself.
\end{lemma}
\begin{proof}
By Lemma \ref{upbound_min_h}, for any $a\in \mathcal{A}$ there exists $\ell \in I$ such that the denominator in \eqref{def_F} is bounded by $e^{-\frac{-\Vert c_\ell\Vert_\infty +\tau}{\sigma}}>0$ for every $x\in X$. Hence, for any $a\in \mathcal{A}$, $F(a)$ is well defined.
    By equation \eqref{def_F} one has for any $a\in \cA$, $x\in X$ and $i\in I$:
\begin{equation*}
   F_i(a)(x)\in [0,1] \text{ and } \sum_{j\in I}F_j(a)(x) = 1.
\end{equation*}
Therefore for any $a\in \cA$, it holds $F(a)\in \cA$.
\end{proof}

\begin{lemma}\label{lem:cont_F}
    Under Assumption \ref{ass:transport_costs} and \ref{ass:existence_interaction_term}, the map $F$ is continuous from $\mathcal{A}$ into itself with respect to the weak topology of $L^2(X,\mu;\mathbb{R}^N)$. 
\end{lemma}
 \begin{proof}
     Let $a$ be an element of $\mathcal{A}$, and $(a^n)_{n\in \mathbb{N}}$ be a sequence in $\cA$ weakly converging to $a$ in $L^2(X,\mu;\mathbb{R}^N)$. 
     Let $(I^\infty,I^0)$ be a partition of $I$ defined by :
     \begin{equation*}
         I^\infty:=\{i\in I\,\vert\,h_i(a_i)=+\infty\}\,\quad\mbox{and}\quad I^0:=\{i\in I\,\vert\,h_i(a_i)<+\infty\}.
     \end{equation*}
     For any $i$ in $I^\infty$, by Assumption \ref{ass:existence_interaction_term} the map $h_i$ is weakly lower-semicontinuous, therefore
     \begin{equation*}
          \liminf_{n\to \infty}h_i( a^n_i)\geq h_i(a_i)=+\infty,
     \end{equation*}
     and thus
          \begin{equation*}
          \lim_{n\to \infty}e^{-\frac{h_i(a^n_i)}{\sigma}} = 0.
     \end{equation*}
     For any $i$ in $I^0$, by Assumption \ref{ass:existence_interaction_term} the map $h_i$ is weakly continuous on its domain that is open w.r.t. the weak topology in $L^2(X,\mu;\mathbb{R}^N)$, thus 
               \begin{equation*}
          \lim_{n\to \infty}e^{-\frac{h_i(a^n_i)}{\sigma}} =e^{-\frac{h_i(a_i)}{\sigma}}.
     \end{equation*}
     Therefore, for any $i\in I$ and any $x\in X$, the point-wise convergence holds
     \begin{equation*}
         \lim_{n\to\infty}F_i(a^n)(x)=F_i(a)(x).
     \end{equation*}
     The domintated convergence yields that $F(a^n)$ strongly converges to $F(a)$, hence weakly, in $L^2(X,\mu;\mathbb{R}^N)$. 
 \end{proof}

\begin{theorem}\label{existence_SUE}
Under Assumptions \ref{ass:transport_costs}-\ref{ass:existence_interaction_term}, for any positive $\sigma$ there exists a SVNE$_\sigma$.
\end{theorem}
\begin{proof}
    From Corollary \ref{cor:SVNE_fixed_point} and equation \eqref{def_svne_gumbel} finding a SVNE$_\sigma$ is equivalent to solving $a=F(a)$. 

    To establish that $F$ admits a fixed-point we use the  Schauder–Tychonoff fixed-point theorem. The set $\cA$ is a nonempty convex closed subset of $L^2(X,\mu;\mathbb{R}^N)$ and is compact w.r.t. the weak topology. By Lemma \ref{F_well_defined}, $F$ is well defined on $\mathcal{A}$, weakly continuous and the image of $\mathcal{A}$ by $F$ is included in $\mathcal{A}$. Thus, the  Schauder–Tychonoff fixed-point theorem yields that $F$ has a fixed-point.
\end{proof}

The following proposition ensures that at equilibrium each congestion terms $h_1(a_1),\ldots,h_N(a_N)$ is finite.

\begin{proposition}\label{finite_time_eq}
    Under Assumption \ref{ass:transport_costs} and \ref{ass:existence_interaction_term}, for any fixed-point $a$ of $F$, it holds for any $i$ that $h_i(a_i)$ is finite.
\end{proposition}
\begin{proof}
The proof is by contradiction. Assume there exists $i$ in $I$ such that $h_i( a_i)=+\infty$. Since $a$ is a fixed point of $F$, by \eqref{def_F} it holds for any $x$ in $X$ :
\begin{equation*}
    a_i(x)= \frac{e^{-\frac{c_i(x) + h_i( a_i)}{\sigma}}}{\displaystyle\sum_{j\in I}e^{-\frac{c_j(x) + h_j( a_j)}{\sigma}}}=0.
\end{equation*}
Thus, since by assumption one has $h_i(0)<+\infty$, $h_i(a_i) = +\infty$ is impossible.
\end{proof}

\subsubsection{The zero-noise limit case}\label{no_noise_eq_section}

From the previous paragraph, the existence of a SVNE$_\sigma$ holds for $\sigma>0$. We now extend the existence result by letting $\sigma$ go to zero. 
The following lemma provides an upper bound on the congestion terms for any SVNE$_\sigma$ and $\sigma>0$.

\begin{lemma}\label{upper_bound_congestion_terms}
Under Assumptions \ref{ass:transport_costs}-\ref{ass:existence_interaction_term}, 
    there exist $\sigma^0>0$ and $M>0$ such that for any $\sigma\in [0,\sigma^0)$, any $a\in \mathcal{A}$ defined in \eqref{def_a} from a SVNE$_\sigma$ satisfies for any $i\in I$
    \begin{equation*}
        \vert h_i(a_i)\vert <M.
    \end{equation*}
\end{lemma}

\begin{proof}
    By Proposition \ref{finite_time_eq}, for any $\sigma>0$ and associated fixed point $a$ of \eqref{fixed_point_F}, $\Vert h(a)\Vert_\infty$ is finite.
    Let $(\sigma^n)_{n\in\mathbb{N}}$ be a monotone sequence of positive real numbers converging to $0$ and for any $n$ in $\mathbb{N}$ $a^n$ be a fixed point of \eqref{fixed_point_F} associated with $\sigma^n$. We proceed by contradiction: we assume that, there exists a subsequence of $(a_n)_{n\in \mathbb{N}}$ denoted in the same way by abuse of notations and a sequence of indices $(i^n)_{n\in \mathbb{N}}$ such that $h_{i^n}(a^n_{i^n})$ is non decreasing, and diverges towards $+\infty$. Since $a^n$ is a fixed point of \eqref{fixed_point_F}, it satisfies for any $x\in X$
    \begin{equation*}
        a^n_{i^n}(x)=\frac{e^{-\frac{c_{i^n}(x) + h_{i^n}(a_{i^n}^n)}{\sigma^n}}}{\displaystyle\sum_{j\in I}e^{-\frac{c_j(x) + h_j(a^n_j)}{\sigma^n}}}.
    \end{equation*}
    By Lemma \ref{upbound_min_h} and the fact that $c$ is non negative (Assumption \ref{ass:transport_costs}.1), 
        \begin{equation*}
        \Vert a^n_{i^n}\Vert_\infty\leq \frac{e^{-\frac{ h_{i^n}(a_{i^n}^n)}{\sigma^n}}}{e^{-\frac{\norm{c}_\infty+\tau}{\sigma^n}}}.
    \end{equation*}
    Previous inequality yields that the sequence $(a^n_{i^n})_{n\in \mathbb{N}}$ uniformly converges to 0. Further, 
    since $(i^n)_{n\in \mathbb{N}}$ takes values in the finite set $I$, one can extract a sub-sequence $(i^{\varphi(n)})_{n\in \mathbb{N}}$ that is constant and equal to $\ell\in I$.    
    Assumption \ref{ass:existence_interaction_term} provides that $h$ is bounded and continuous in a neighbourhood of $0$ w.r.t. to the weak topology in $L^2(X,\mu;\R^N)$ then,
    \begin{equation*}
        \underset{n\to\infty }{\lim}h_{i^{\varphi(n)}}\big(a_{i^{\varphi(n)}}^{\varphi(n)}\big)= \underset{n\to\infty }{\lim}h_{\ell}\big(a_{\ell}^{\varphi(n)}\big)= h_\ell(0)<\infty,
    \end{equation*}
    which contradicts the divergence of $(h_{i^n}(a^n_{i^n}))_n$ towards $+\infty$ as $\sigma^n$ goes to $0$.

\end{proof}

\begin{lemma}
    \label{limit_point}
    Let $(\sigma_n)_{n \in \mathbb{N}}$ be a sequence of positive real numbers converging to $0$ and let $(a^{n})_{n\in \mathbb{N}}$ be a sequence such that for any $n$, $a^n$ is associated with the SVNE$_\sigma$ as in \eqref{def_a}. \newline Under Assumptions \ref{ass:transport_costs}-\ref{ass:existence_interaction_term}, if $a^\infty$ is an accumulation point of $(a^n)_{n\in \mathbb{N}}$ w.r.t. the weak topology of $L^2(X,\mu;\mathbb{R}^N)$,  then for $\mu$-a.e. $x\in X$ and any $i\in I$ 
    \begin{equation*}
      a^{\infty}_i(x)=\ind_{A_i}(x), 
    \end{equation*}
    where
\begin{equation}\label{def_set_A}
    A_i:=\big\{x\in X\,\vert\,c_i(x)+h_i( a_i^{\infty})=\min_{j\in I}c_j(x)+h_j( a_j^{\infty})\big\}.
\end{equation}
\end{lemma}
\begin{proof}
    Without loss of generality let us assume that $a^n$ weakly converges to $a^\infty$ in $L^2(X,\mu;\mathbb{R}^N)$. By Lemma \ref{upper_bound_congestion_terms}, the sequence $(h(a_n))_n$ is uniformly bounded. Then, by Assumption \ref{ass:existence_interaction_term} the sequence $(h(a_n))_n$ converges to $h(a^\infty)$ that is finite.
    Then, the no-tie assumption in Assumption \ref{ass:transport_costs}
    and the point-wise convergence of the \textit{softmin operator} \eqref{pointwise_conv}
    yield the point-wise convergence $\mu$-a.e. on X of the sequence $(a^n_i)_{n\in \mathbb{N}}$ to the map
\begin{equation}\label{point_wise_conv}
    x\mapsto    \left\{
        \begin{array}{ll}
          1   & \mbox{if }c_i(x)+h_i( a_i^{\infty})=\min_{j\in I}c_j(x)+h_j( a_j^{\infty}) \\
          0   & \mbox{else}.
        \end{array}
        \right.
    \end{equation}
    Since the sequence $(a^n)_{n\in \mathbb{N}}$ is in $\cA$, it is uniformly bounded and then, the dominated convergence implies that the sequence $(a_i^n)_{n\in \mathbb{N}}$ strongly-converges in $L^2((X,\mathbb{R}^N)$ to the indicator function defined in \eqref{point_wise_conv}.
    The uniqueness of the weak limit leads to 
\begin{equation*}
    a_i^{\infty}(x)=\ind_{A_i}(x),\quad\text{for }\mu-a.e. \,x\in X.
\end{equation*}

\end{proof}

\begin{theorem}\label{thm:existence_PNE}
    Under Assumptions \ref{ass:transport_costs}-\ref{ass:existence_interaction_term}, there exist a SVNE$_0$ and a PNE. 
    
\end{theorem}
\begin{proof}
    Let us consider a sequence $(\sigma_n)_{n\in \mathbb{N}}$ converging to 0. For any $n\in \mathbb{N}$ we consider $a^n$ in $\cA$ to be an associated SVNE$_{\sigma_n}$.
    Since $\cA$ is compact for the weak topology of $L^2$, the sequence $(a_n)_{n\in \mathbb{N}}$ is a weakly compact in $L^2(X,\mu;\mathbb{R}^N)$. 
    
    Let $a^\infty\in\cA$ be an accumulation point as in Lemma \ref{limit_point}, and $(A_1,\ldots,A_N)$ be the associated set defined in \eqref{def_set_A}.
    For any $i\in I$, $\varepsilon\in \mathbb{R}^N$, $x\in X$, let $f_i(\varepsilon,x):=a_i^\infty(x)$. It is easy to check that $f$ belongs to the set $\cF$ defined in Definition \ref{def:SVNE}. Then, by Proposition \ref{prop:justification_SVNE} $f$ is a SVNE$_0$.
    
     On the other hand, by Lemma \ref{lem:intersection_mesure_zero} we note that $(A_1,...,A_N)$ is a partition up to $\mu$-null sets. Corollary \ref{cor:equivalence} yields that $(A_1,...,A_N)$ is a PNE.
\end{proof}

\subsection{Uniqueness}\label{uniqueness_section}
In this section we show for any $\sigma\geq 0$ the uniqueness of a SVNE$_\sigma$ under a monotony assumption of the operator $h$. Our strategy is to show that the congestion terms are unique, to deduce that the SVNE is.

\begin{lemma}[Uniqueness of aggregates]\label{lem:uniqueness_ag}
    Under Assumptions \ref{ass:transport_costs}-\ref{ass:uniqueness_interaction_term}, for any $\sigma \geq 0$ there exists $T\in \mathbb{R}^N$ such that for any SVNE$_\sigma$, named $f$, it holds
    \begin{equation*}
        h(a)=T,
    \end{equation*}
   where $a_i(x) := \int_{\R^N} f_i(\varepsilon,x)d\nu^N(\varepsilon)$ for any $(x,i)\in X\times I$.
\end{lemma}
\begin{proof}
 We argue by contradiction. Assume there exist two SVNE$_\sigma$ $f^1$ and $f^2$ with associated $a^1$ and $a^2$, and two vectors $T^1,T^2\in \mathbb{R}^N$ such that 
 \begin{equation*}
  h(a^1)=\tilde h(\Psi a^1)=T^1\quad \text{ and }\quad  h(a^2)=\tilde h(\Psi a^2)=T^2,
 \end{equation*}
 and $T^1\neq T^2$. Let for any $i$ in $I$, $\delta_i$ be defined by 
\begin{equation*}
    \delta_i:=T^{1}_i-T^{2}_i.
\end{equation*}
Without loss of generality, we assume there exists $p\in I$ such that
  \begin{equation}\label{def_p_2}
      \delta_p = \min_{i\in I}\delta_i<0.
   \end{equation}
Then, it holds for any Lebesgue point $(\varepsilon,x)$ of $\nu^N\otimes \mu$, $f^1$ and $f^2$
such that $f^2_p(\varepsilon,x)>0$ and for any $j\in I\setminus \{p\} $  that
\begin{equation*}
    \begin{array}{ll}
    c_p(x)+T^1_p-\sigma\varepsilon_p     & = c_p(x)+T^2_p +\delta_p -\sigma\varepsilon_p\vspace{0.1cm}  \\
         & \leq c_j(x)+T^2_j +\delta_p-\sigma\varepsilon_j \vspace{0.1cm}\\
         &\leq c_j(x)+T^1_j - \delta_j+\delta_p -\sigma\varepsilon_j\vspace{0.1cm}\\
         &\leq c_j(x)+T^1_j-\sigma\varepsilon_j.
    \end{array}
\end{equation*}
Therefore, for $\nu^N\otimes \mu$-a.a. $(\varepsilon, x)$ in $(f_p^2)^{-1}(\R_+^*)$, $(\varepsilon,x)$ belongs to  
\begin{displaymath}
    \{(\epsilon,y)\;:\; c_p(y) + T^1_p -\sigma \epsilon_p =\min_i c_i(y) + T^1_i -\sigma \epsilon_i\}.
\end{displaymath}
In view of Lemma \ref{lem:intersection_mesure_zero} and Proposition \ref{prop:voronoi_type_result}, we observe that $f_p^1(\varepsilon, x) = 1$ for $\nu^N \otimes \mu$-almost every $(\varepsilon, x)$ in the support of $f_p^2$. Furthermore, since the mappings $(\psi^k)_{k}$ are assumed to be positive according to Assumption \ref{ass:uniqueness_interaction_term}, it follows that for any $\ell \in \{1, \dots, m\}$:
\begin{equation*}
 \int_{\R^N\times X}\psi^\ell(x)f^2_p(\varepsilon,x)d\nu^N\otimes \mu(\varepsilon, x) \leq  \int_{\R^N\times X}\psi^\ell(x)f^1_p(\varepsilon,x)d\nu^N\otimes \mu(\varepsilon, x) .
\end{equation*}
Previous inequalities and Assumption \ref{decreasing} yields 
\begin{equation*}
\tilde{h}_p(\Psi a_p^2)\leq \tilde{h}_p(\Psi a_p^1),
\end{equation*}
which contradicts \eqref{def_p_2}. 
\end{proof}

\begin{theorem}\label{thm:uniqueness}
If Assumptions \ref{ass:transport_costs}-\ref{ass:uniqueness_interaction_term} holds, then for any $\sigma\geq 0$, there exists, up to $\nu^N\otimes \mu$-null sets, a unique SVNE$_\sigma$.
\end{theorem}
\begin{proof}
    From Proposition \ref{prop:voronoi_type_result} and Lemma \ref{lem:uniqueness_ag}, 
   $\nu^N\otimes \mu$-a.s.
        \begin{equation*}
            f_i(\varepsilon,x)=
                \begin{cases}
                1, & i \in \arg\min_{j}\left\{c_j(x)+T_j-\sigma \varepsilon_j\right\},\\
                0, & \text{otherwise}.
                \end{cases}
            \end{equation*}
            Since $T$ is uniquely determined, so does $f$. 
\end{proof}
We are now ready to prove the main results of the paper.

\begin{proof}[Proof of Theorem \ref{main_result}]
Theorem \ref{thm:uniqueness} provides existence and uniqueness of the SVNE$_\sigma$ for any $\sigma\geq 0$. By Theorem \ref{thm:existence_PNE} a PNE exists. By Corollary \ref{cor:equivalence}, the uniqueness of the SVNE$_0$ implies the uniqueness of the PNE up to $\mu$-null sets.
\end{proof}

\begin{proof}[Proof of Theorem \ref{thm:conv}]
By Theorem \ref{thm:uniqueness} there exists a unique PNE $(A_1,\ldots,A_N)$. By Lemma \ref{limit_point} and the uniqueness of the PNE, for any $i\in I$ the sequence $h_i(a_i^\sigma)$ converges to $h_i(\ind_{A_i})$ as $\sigma$ goes to $0$. By the definition of $B_i^\sigma$ in \eqref{def_B_sigma}, the sequence of indicator functions $(\ind_{B_i^\sigma})_\sigma$ converges point-wise $\mu$-a.e. on $X$ to $\ind_{A_i}$. The dominated convergence yields for any $i\in I$ the following limit
\begin{equation*}
    \lim_{\sigma\to 0}\mu(B^\sigma_i\Delta A_i)=\lim_{\sigma\to 0}\int_X\vert \ind_{B^\sigma_i}(x)-\ind_{A_i}(x)\vert d\mu(x)=0
\end{equation*}
\end{proof}

\section{Numerical simulations}\label{illustrative_ex}

By Corollary \ref{cor:SVNE_fixed_point}, solving equation \eqref{fixed_point_F} is equivalent to finding SVNE$_\sigma$. Applying $h$ to both sides gives
\begin{displaymath}
    h(a) = h(F(a)).
\end{displaymath}
Using the definition of $F$ in \eqref{def_F}, together with Proposition \ref{prop:voronoi_type_result}, we deduce that finding a SVNE is equivalent to finding an aggregate cost vector $T$ satisfying
\begin{equation}
    \label{eq:fixed_point_T}
    T = h\left[(\alpha(T,\cdot)\right],
\end{equation}
where for any $(x,T)\in \Omega\times \mathbb{R}^+$
\begin{displaymath}
    \alpha(T,x) :=\left( \frac{e^{-\frac{c_i(x) + T_i}{\sigma}}}{\displaystyle\sum_{j\in I}e^{-\frac{c_j(x) + T_j}{\sigma}}}\right)_{i=1,...,N}.
\end{displaymath}

Here, $\alpha(T,x)$ is the best response of a player with characteristics $x$ against waiting times $T$.

\subsection{The scheme: discretization of the continuous set}

In practice, $h$ is of the form given in Assumption \ref{ass:uniqueness_interaction_term}. Therefore, there exist $\tilde h$, and $\psi^1$ up to $\psi^m$ such that, for any $a$,
\begin{displaymath}
    h(a) =\big(\tilde h_1(\Psi  a_1),\ldots,\tilde h_N(\Psi  a_N)\big),
\end{displaymath}
where the linear map $\Psi:L^2(X,\mu;\R)\rightarrow \R^m$ is defined by:
\begin{equation*}
    \Psi f :=\left(\int_X \psi^1(x)\,f(x) d\mu(x), \ldots, \int_X \psi(x)^m\,f(x)d\mu(x)\right).
\end{equation*}
For simplicity, we describe the discretization method on the one-dimensional case $X = [0,1]$. The scheme, however, can be extended in a straightforward way to the case when $X$ is a bounded domain of $\R^d$, $d>1$. Let $X_h$ be a uniform grid on $X$ with step $\delta = 1/N_\delta$, $N_\delta\in\N$. Let $x_k$ denote a generic point in $X_h$; the values of $c_i$, $\alpha_i(\cdot,T)$, $\psi^i$, and $\mu$ at $x_k$ will be respectively denoted by $c_i^k$, $\alpha^k_i(T)$, $\psi^i_k$, and $\mu_k$ and computed by the explicit form of $c_i$, $\alpha_i(\cdot,T)$, $\psi^i$, and $\mu$. We use the trapezoidal rule to estimate the value of the integrals. Therefore, we make the following approximation: 
\begin{equation}\label{eq:scheme_for_fixed_point_T}
    \int_X \psi^j(x)\alpha_i(T,x)d\mu(x) \simeq \delta \left( \frac{\psi^j_0 \alpha_i^0\mu_0 + \psi^j_{N_\delta} \alpha_i^{N_\delta}\mu_{N_\delta}}{2} +\sum_{k=1}^{N_\delta - 1}\psi^j_k \alpha_i^k\mu_k\right).
\end{equation}
In higher dimension, the same construction is applied using the multidimensional trapezoidal rule. We denote by $\Psi^\delta : \R^{N_\delta+1}\rightarrow \R^m$ the corresponding discrete approximation of $\Psi$.
Therefore, we aim to solve
\begin{equation}\label{eq:discretized_H}
    T = \tilde h\left[\Psi^\delta(\alpha^1_1(T,\cdot),...,\alpha^{N_\delta}_1(T,\cdot)),\;... ,\;\Psi^\delta(\alpha^1_n(T,\cdot),...,\alpha^{N_\delta}_n(T,\cdot)) \right].
\end{equation}

\subsection{The method} \label{sec:the_method}

The discrete fixed-point problem \eqref{eq:scheme_for_fixed_point_T} is solved by rewriting it as the nonlinear system
\begin{displaymath}
    G_\delta(T):=T-H_\delta(T)=0,
\end{displaymath}
where $H_\delta$ is the right-hand side of \eqref{eq:discretized_H}.

Since $H_\delta$ may take infinite values, one can not directly apply quasi-Newton algorithms.

Using Proposition~\ref{finite_time_eq}, the interaction costs are finite at any solution of \eqref{eq:fixed_point_T}. We therefore regularize $H_\delta$ using a smooth truncation with an upper bound $K\in\mathbb{R}_+^*$, treated as an adaptive numerical parameter to ensure that the truncation remains inactive at equilibrium. More precisely, after solving the truncated problem, we verify a posteriori that the truncation is inactive at the computed root. If it is active, the value of $K$ is increased and the problem is solved again. Once the truncation is inactive, the computed root also solves the original discrete fixed-point problem.

We then apply the good Broyden method to find zeros of $G_\delta$ (see, for instance, \cite{Broyden1965} or \cite{DennisSchnabel1996}). This method belongs to the class of quasi-Newton methods: instead of computing the Jacobian of $G_\delta$ explicitly, it constructs and updates an approximation of it along the iterations.
It is particularly useful in the present setting. Indeed, although the dimension of the unknown vector $T$ is only $N$, the map $H_\delta$ involves the evaluation of the truncation, the interaction term $\tilde h$ and the computation of integrals. Therefore, its derivative may be numerically intensive to compute explicitly. The good Broyden method therefore provides an efficient derivative-free alternative to Newton's method with superlinear local convergence when the Jacobian approximation is sufficiently accurate.

\paragraph{Numerical parameter values.}

The good Broyden method is initialized at $T_0 = 0_{\mathbb{R}^N}$, corresponding to zero initial congestion costs. The algorithm is stopped at the first iteration $n$ such that \[ \lVert G_\delta(T_n) \rVert_\infty \leq 10^{-12}. \] The upper truncation bound is set to $K=100$. For all the equilibria reported in Section~\ref{sec:use_case}, each component of the computed interaction-cost vector is strictly below $K$. Hence, the truncation is inactive at the computed solutions, which therefore also solve the original discrete fixed-point problem.

\subsection{An illustrative use case: electric vehicle charging}
\label{sec:use_case}

We now illustrate the model introduced in
Section~\ref{congestion_game_def} with an application to electric vehicle
(EV) charging. The characteristic space is
\begin{displaymath}
X=[0,10]\times[0,50],
\end{displaymath}
where \(x=(y,e)\in X\) describes an EV user, i.e., a player, through its location
\(y\in[0,10]\) and its energy demand \(e\in[0,50]\), expressed in
\(\mathrm{kWh}\). 
User characteristics are uniformly distributed on
\(X\), that is, \(\mu=\mathcal U(X)\). Let $\pi_e:X\mapsto [0,50]$ be the canonical projection of any $x=(y,e)$ in $X$ to its second coordinate $e$
\begin{equation*}
    \pi_e(x):=e.
\end{equation*}
We consider two Charging Stations (CSs), indexed by \(i\in\{0,1\}\), located
at \(z_0=3\) and \(z_1=7\), respectively. The idiosyncratic cost
incurred by a user \(x=(y,e)\) choosing CS \(i\) is
\begin{equation}
\label{eq:ev_idiosyncratic_cost}
c_i(x)
=
r_i\,e + d_i\lvert y-z_i\rvert,
\end{equation}
where \(r_i\) is the price per unit of energy, expressed in \(\text{\euro}/\text{kWh}\), and \(d_i\) is a
transportation cost coefficient, expressed in \(\text{\euro}/\text{km}\). The latter accounts for both the energy
required to reach the CS and the monetary valuation of the
corresponding travel time.

For any measurable strategy \(a_i:X\to[0,1]\), define the two normalized aggregate
energy-demand moments
\begin{equation}
\label{eq:ev_aggregates}
\Psi a_i
:=
\begin{pmatrix}
\Psi_1 a_i\\
\Psi_2 a_i
\end{pmatrix}
:=
\begin{pmatrix}
\displaystyle\int_X \pi_e(x)\,a_i(x)\,d\mu(x)\\[0.8em]
\displaystyle\int_X \pi_e(x)^2\,a_i(x)\,d\mu(x)
\end{pmatrix}.
\end{equation}

Let \(\lambda>0\) denote the total flow mass of EV users, expressed in $h^{-1}$. In this application, the population mass is thus interpreted as an arrival-flow intensity. Consequently, $\lambda\mu$ is defines a finite flow-mass measure: for any measurable set $A \subseteq X, \lambda\mu(A)$ is the mean number of EV users with characteristics in $A$ arriving per unit of time.

Although the theoretical framework accommodates an arbitrary number of
charging points, we assume for the numerical illustration that each
CS is equipped with two charging points, each delivering a charging
power \(p_i>0\).  

Based on the multi-server queueing approximation proposed in
\cite{Nozaki1978_JAP}, and similarly to \cite{Yuan2015_TSG}, the congestion cost $h_i(a_i)$ at CS \(i\) is defined
as an approximation of the stationary mean waiting time in an
\(M/G/2\) queue\footnote{In Kendall's notation, \(M\) indicates Poisson
arrivals, \(G\) an arbitrary service-time distribution, and \(2\) the
number of parallel servers, corresponding here to the two charging
points.} converted in monetary units with the parameter $\omega$ (in \euro/h):
\begin{equation}
\label{eq:ev_waiting_cost}
h_i(a_i)
=
\begin{cases}
\displaystyle
\frac{
\omega\,
\lambda^2
\left(\Psi_1 a_i\right)
\left(\Psi_2 a_i\right)
}{
2p_i
\left(
4p_i^2-\lambda^2\left(\Psi_1 a_i\right)^2
\right)
},
&
\lambda\Psi_1 a_i<2p_i,
\\[0.8em]
+\infty,
&
\lambda\Psi_1 a_i\geq2p_i.
\end{cases}
\end{equation}
The quantity \(\Psi_1 a_i\) represents the aggregate energy demand, expressed in kWh, of the normalized population directed towards CS \(i\).
Accordingly, the quantity \(\lambda\Psi_1 a_i\), expressed in kW, represents the aggregate energy-demand rate directed towards CS \(i\), while \(2p_i\) is its total
charging capacity. Accordingly, the condition
\begin{displaymath}
\lambda\Psi_1 a_i<2p_i
\end{displaymath}
is the stability condition of CS \(i\). The waiting cost increases
with the aggregate energy demand and diverges as the CS approaches
saturation.

Moreover, the congestion cost has the aggregate form required in
Assumption~\ref{ass:uniqueness_interaction_term}. Indeed,
\begin{displaymath}
h_i(a_i)=\widetilde h_i(\Psi a_i),
\end{displaymath}
where
\begin{equation}
\label{eq:ev_aggregate_waiting_cost}
\widetilde h_i(u,v)
=
\begin{cases}
\displaystyle
\frac{\lambda^2uv}
{2p_i\left(4p_i^2-\lambda^2u^2\right)},
&
\lambda u<2p_i,
\\[1.2em]
+\infty,
&
\lambda u\geq2p_i.
\end{cases}
\end{equation}
The map \(\widetilde h_i\) is coordinate-wise non-decreasing on its
domain. Hence, the interaction cost satisfies the monotonicity
requirement in
Assumption~\ref{ass:uniqueness_interaction_term}.

\paragraph{Absence of a potential and equilibrium existence and uniqueness.}

The interaction cost depends on the two aggregation kernels \[ \psi_1(y,e)=e \qquad\text{and}\qquad \psi_2(y,e)=e^2. \]
Moreover, \[ \mathbf 1 \notin \operatorname{span}\{\psi_1,\psi_2\} \quad \mu\text{-a.e. on }X. \]
Although the maps $\widetilde h_i$ are not finite on the whole space $\mathbb R^2$, Remark~\ref{rem:abs_potential_using_dom} allows the results of Section~3.1 to be applied on the effective domain of the total cost operator, provided that this domain is convex.
In the present setting, this effective domain is \[ D_h := \left\{ a\in\mathcal A: \lambda\Psi_1a_i<2p_i \text{ for every }i\in I \right\}, \] which is convex.
Furthermore, on its effective domain, each map $\widetilde h_i$ is continuously differentiable and coordinate-wise non-decreasing.
It is also strictly increasing with respect to its second argument whenever $u>0$, since \[ \partial_v\widetilde h_i(u,v) = \frac{\omega\lambda^2u} {2p_i\bigl(4p_i^2-\lambda^2u^2\bigr)} >0 \qquad \text{whenever }u>0\text{ and }\lambda u<2p_i. \] Therefore, Proposition~\ref{prop:aggregate}, together with Remark~\ref{rem:abs_potential_using_dom}, shows that the game does not admit a potential. However, Theorem \ref{main_result} yields the existence and uniqueness of a PNE, as well as,s for any $\sigma\ge 0$, of a SVNE$_\sigma$.

\paragraph{Role of energy-demand heterogeneity.}

In this setting, players contribute heterogeneously to congestion through their
energy demand, but not through their location. While location affects the
transportation cost, a larger energy demand entails a longer service time and
thus generates a greater congestion externality for other EV users.
This differs from previous works such as
\cite{Yuan2015_TSG,Liu2023_IJST,Zhou2021_AE}, which model users'
charging-station choices and the resulting congestion, but do not explicitly
account for the effect of heterogeneous energy demands on both station choice
and waiting times. In \cite{Yuan2015_TSG}, service-time variability is
represented by an exogenous distribution, whereas \cite{Liu2023_IJST} assumes
deterministic charging times for each charger type. Similarly,
\cite{Zhou2021_AE} represents charging demand using a common charging-power
parameter. Consequently, congestion in these models is driven by the aggregate
flow-mass of users choosing a station, i.e. $\int_X a_i(x)\,d\mu(x)$, whereas in our framework it also
depends on the composition of this population in terms of energy demand.
\newline

The parameters used in the numerical experiments are reported in
Table~\ref{tab:station_parameters}. Note that the two CSs differ only in their energy prices $r_i$. Since they are symmetrically located and players are uniformly distributed over the spatial domain, neither CS has a locational advantage.

\begin{center}
    \begin{tabular}{|c|c|c|}
        \hline
        Parameter
        & Station 0 (\(i=0\))
        & Station 1 (\(i=1\))
        \\
        \hline
        \(r_i\) (\(\text{\euro}/\mathrm{kWh}\))
        & \(0.50\)
        & \(0.35\)
        \\
        \hline
        \(p_i\) (\(\mathrm{kW}\))
        & \(150\)
        & \(150\)
        \\
        \hline
        \(z_i\)
        & \(3\)
        & \(7\)
        \\
        \hline
        \(d\) (\(\text{\euro}/\mathrm{km}\))
        & \multicolumn{2}{c|}{\(1.5\)}
        \\
        \hline
        \(\omega\) (\(\text{\euro}/\mathrm{h}\))
        & \multicolumn{2}{c|}{\(30\)}
        \\
        \hline
        \(\lambda\) (\(\mathrm{h}^{-1}\))
        & \multicolumn{2}{c|}{\(10\)}
        \\
        \hline
    \end{tabular}
    \captionof{table}{Parameter values used in the numerical experiments.}
    \label{tab:station_parameters}
\end{center}

In the paragraphs below, we first let $\sigma$ tend to zero, and see the apparition of the expected Voronoï cells characterising a PNE. 
In the second paragraph, we conduct a sensitivity analysis with respect to the mass of players $\lambda$ to evaluate the influence of the congestion on the system.

\paragraph{Parameter continuation: warm start}

When varying $\sigma$ and $\lambda$ in the numerical experiments below, we use a warm-start strategy. The values of $\sigma$ are considered in decreasing order, starting from the largest value, whereas the values of $\lambda$ are considered in increasing order, starting from the smallest value. For each parameter sweep, the equilibrium computed at the previous parameter value is used as the initial iterate of the good Broyden method at the next value.
\vspace{0.5em}

\paragraph{Numerical discretization of the total cost} The total cost $x \mapsto c_i(x) + h_i(a_i(x))$ is evaluated numerically on a uniform $1000 \times 1000$ grid over $X=[0,10]\times[0,50]$. This corresponds to grid spacings of $0.01$ km for the users' locations and $0.05$ kWh for their energy demands.

\subsubsection{Convergence to the PNE as $\sigma$ decreases}

We make a sensibility study and let $\sigma$ goes to 0. From the proof of Theorem \ref{thm:existence_PNE} the sequence of $\mathrm{SVNE}_\sigma$ must converge to the $\mathrm{SVNE}_0$ defining a PNE. Since there are two CSs, we only display in Figure \ref{fig:sensitivity_sigma}, the probability of players to choose the CS 1 depending of there characteristic: we observe the expected convergence toward a PNE.
\begin{center}
    \includegraphics[scale=0.25]{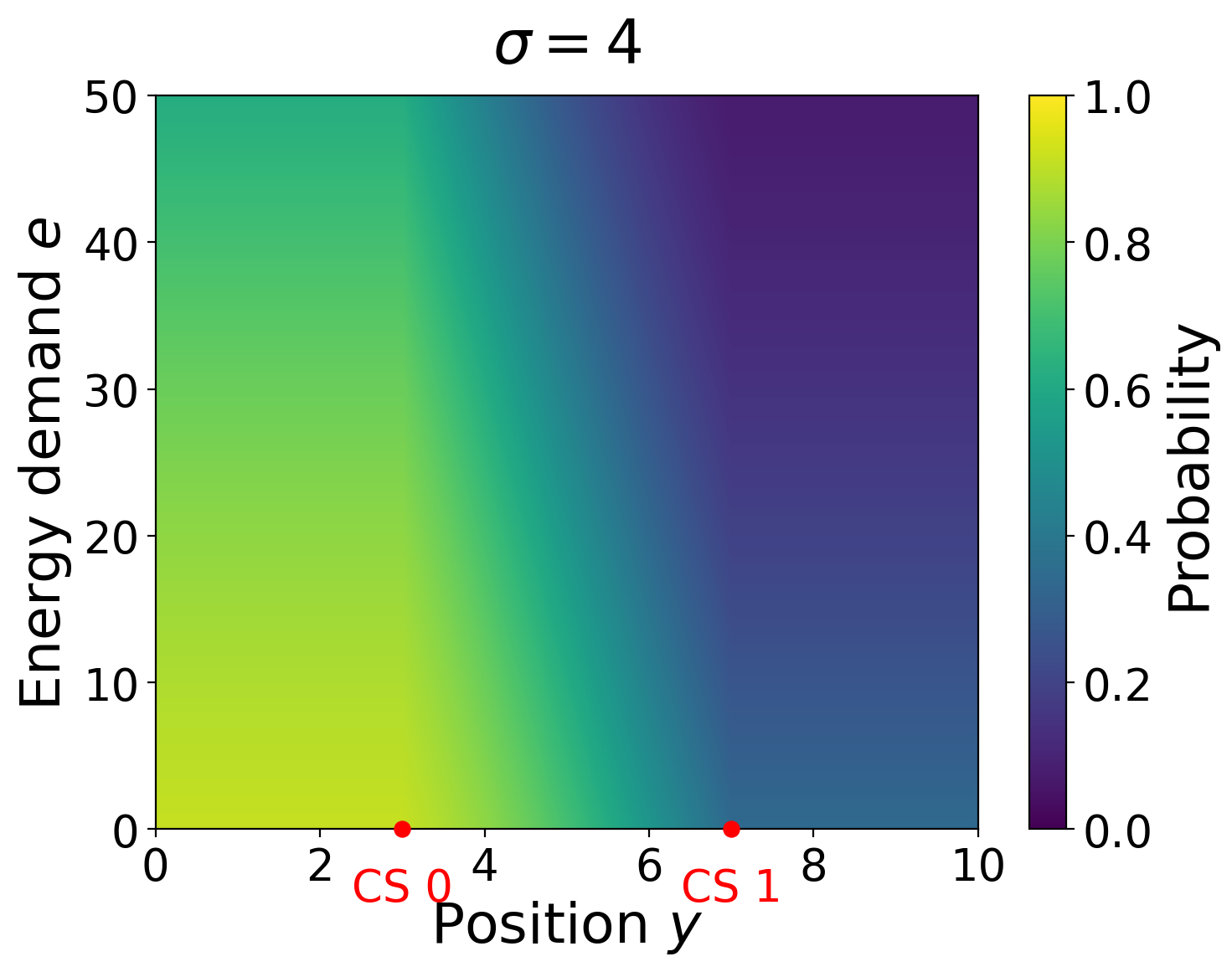}
    \includegraphics[scale=0.25]{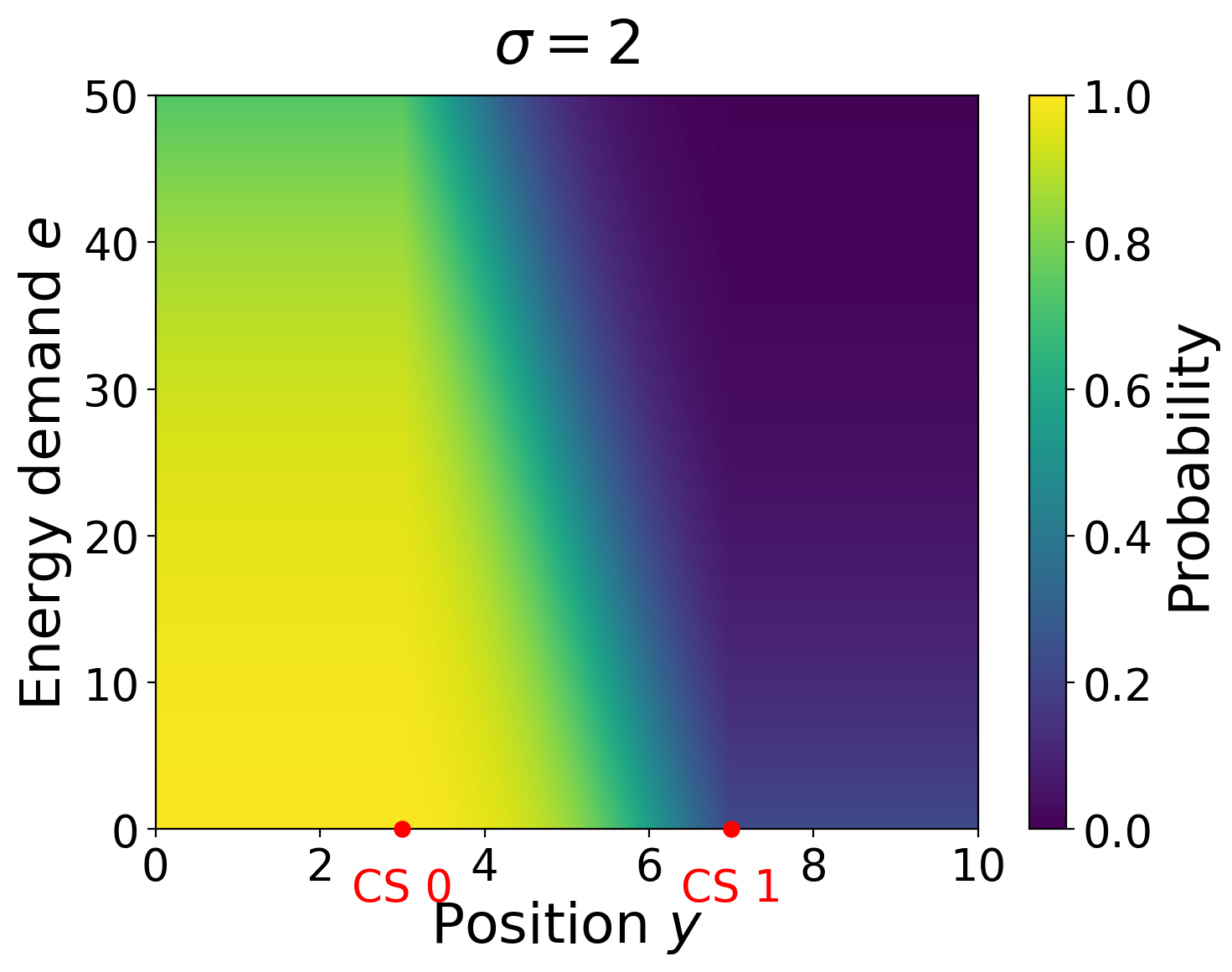}
    \includegraphics[scale=0.25]{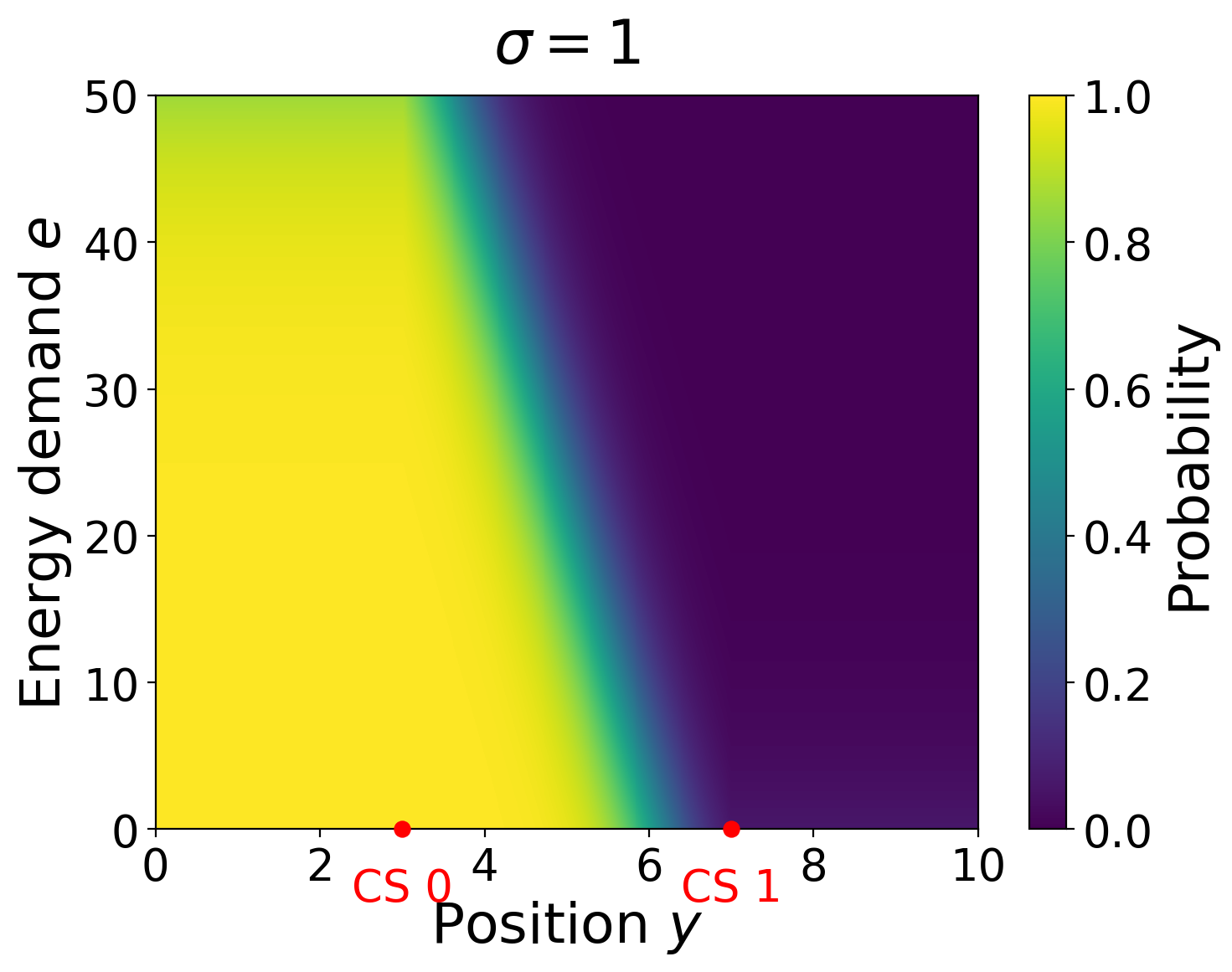}
    \includegraphics[scale=0.25]{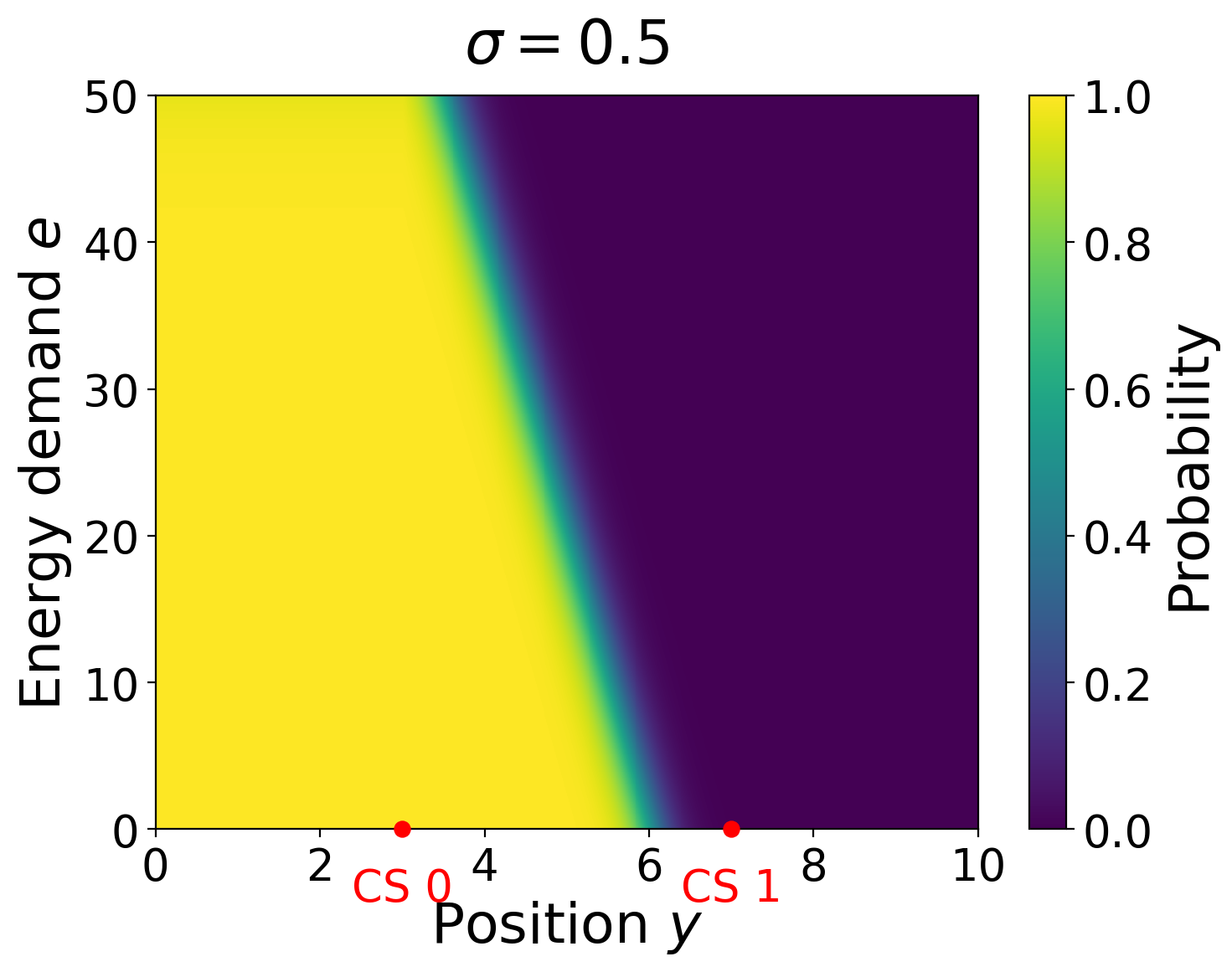}
    \includegraphics[scale=0.25]{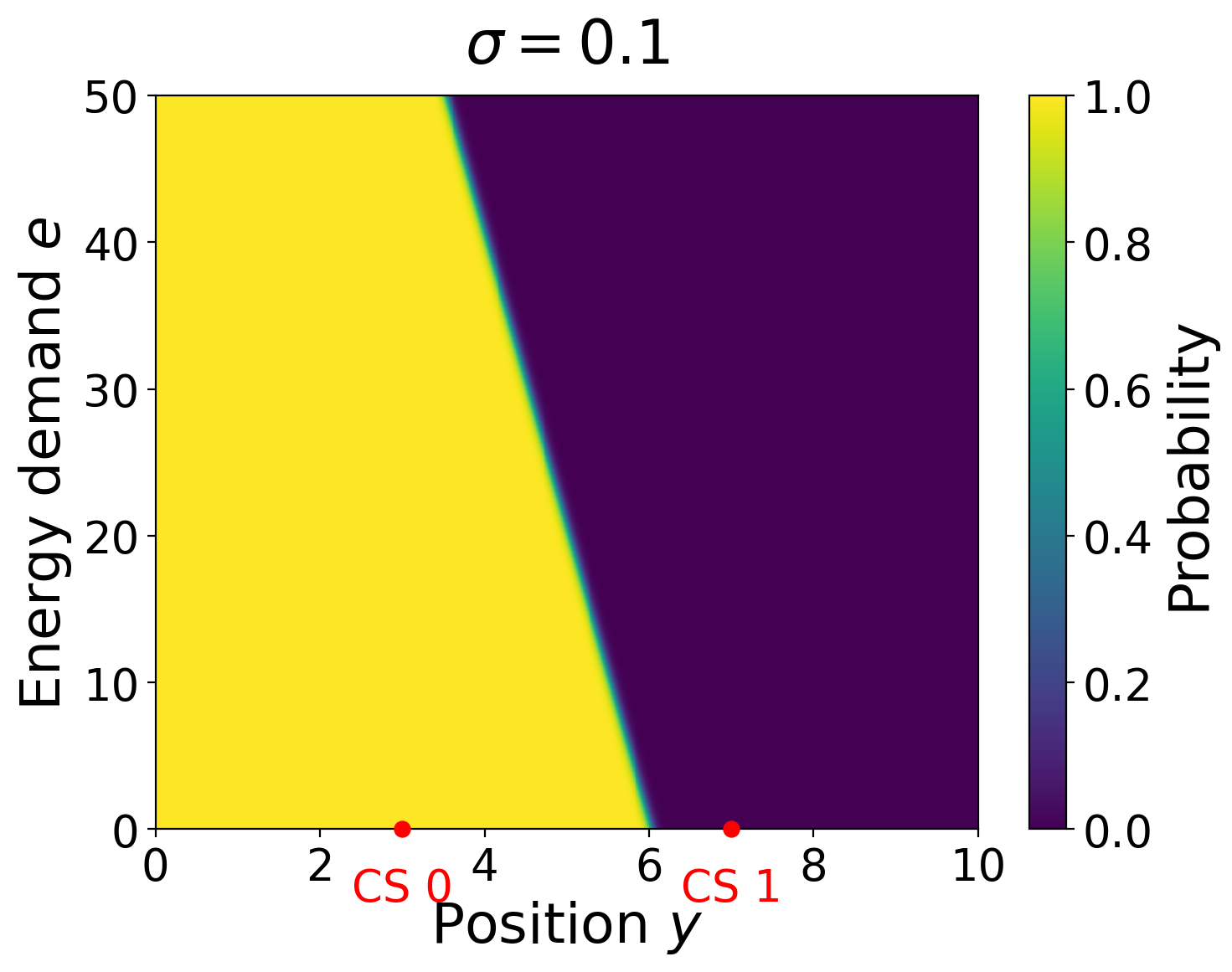}
    \includegraphics[scale=0.25]{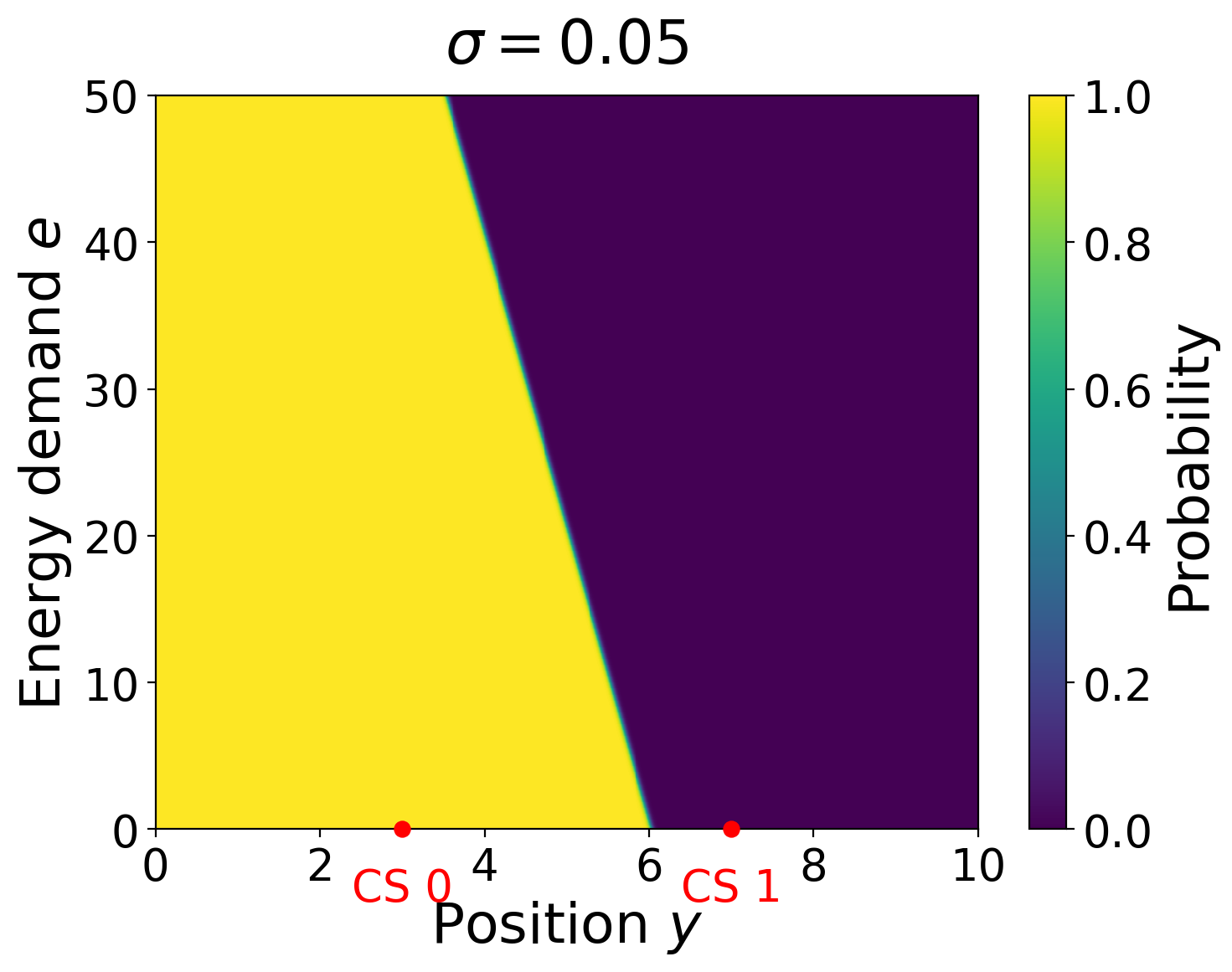}
    \captionof{figure}{\label{fig:sensitivity_sigma}Probability of choosing CS 1 for different values of $\sigma$.}
\end{center}

\subsubsection{Comparative statics with respect to  \texorpdfstring{$\lambda$}{lambda}}

In the following, the parameter $\sigma$ is set to $0.005$. A value close to $\sigma=0$ is chosen in order to approximate the pure Nash equilibrium. 
We now relax the assumption that $\lambda=10$, and perform a comparative statics in which the total flow mass $\lambda$ varies between $10$ and $20$.
This analysis enable to observe how congestion influences the users' choices.
Figure \ref{fig:proba_to_go_to_CS1_mass} below displays the probability of choosing CS 0 for different values of $\lambda$.
\begin{center}
    \includegraphics[scale=0.25]{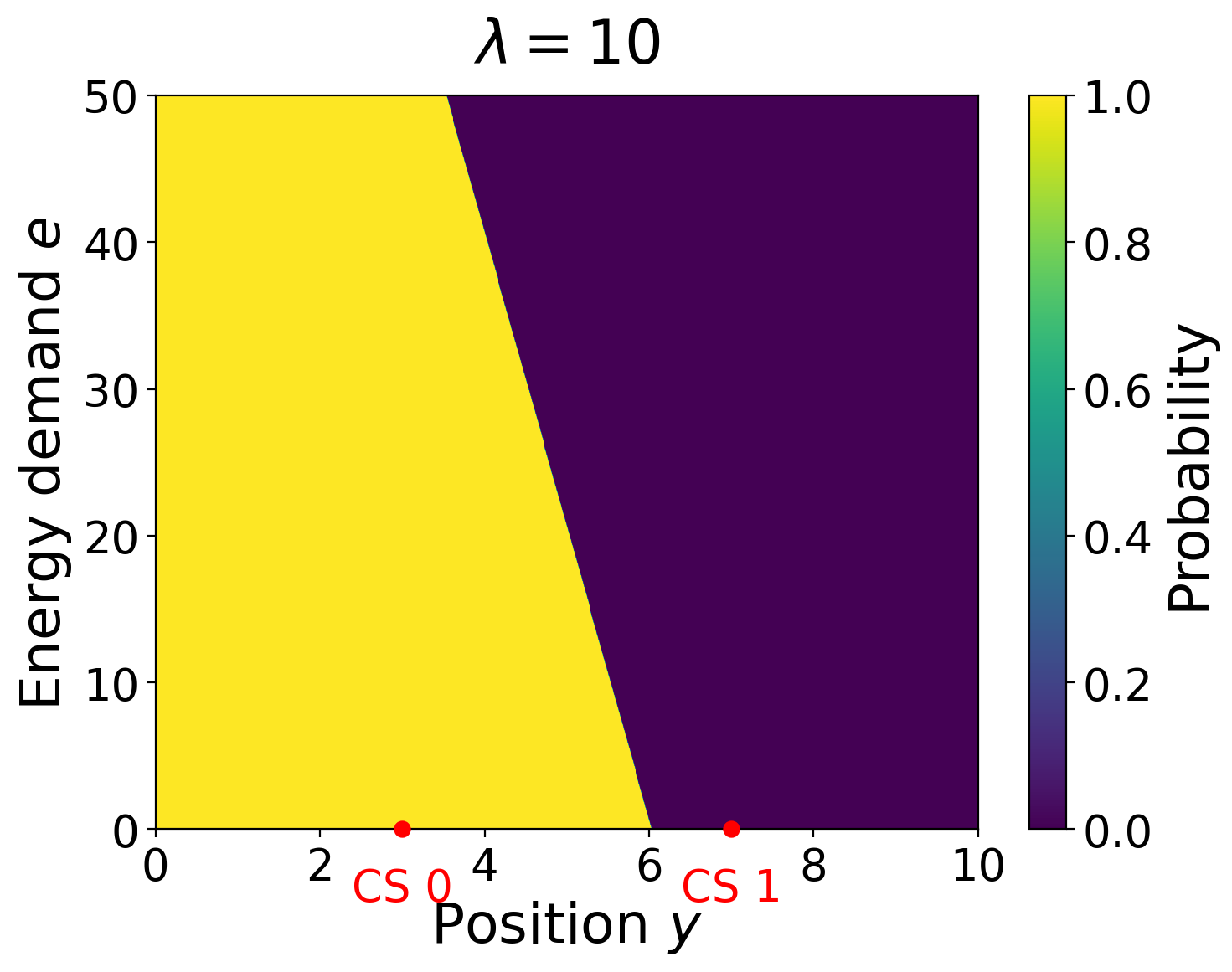}
    \includegraphics[scale=0.25]{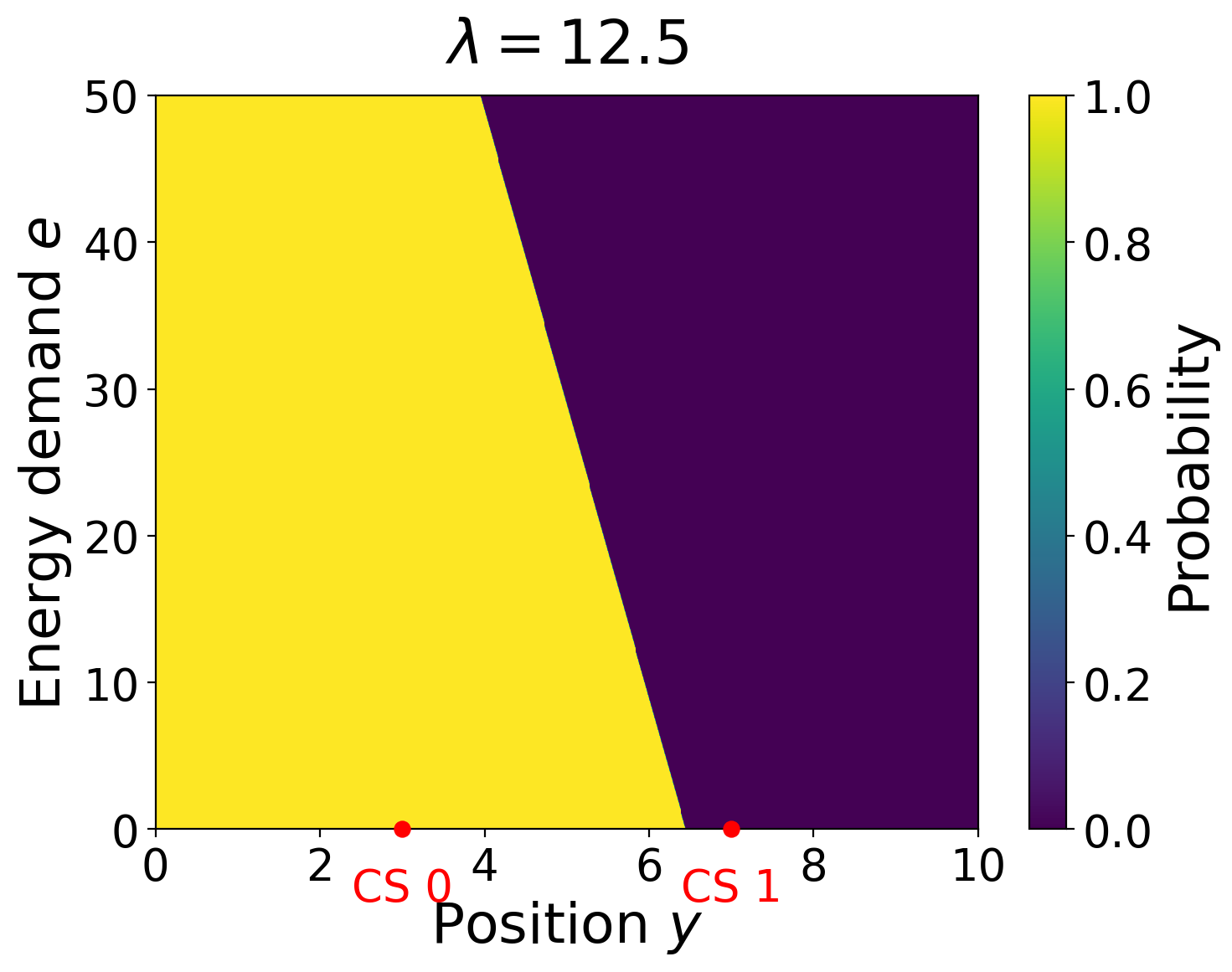}
    \includegraphics[scale=0.25]{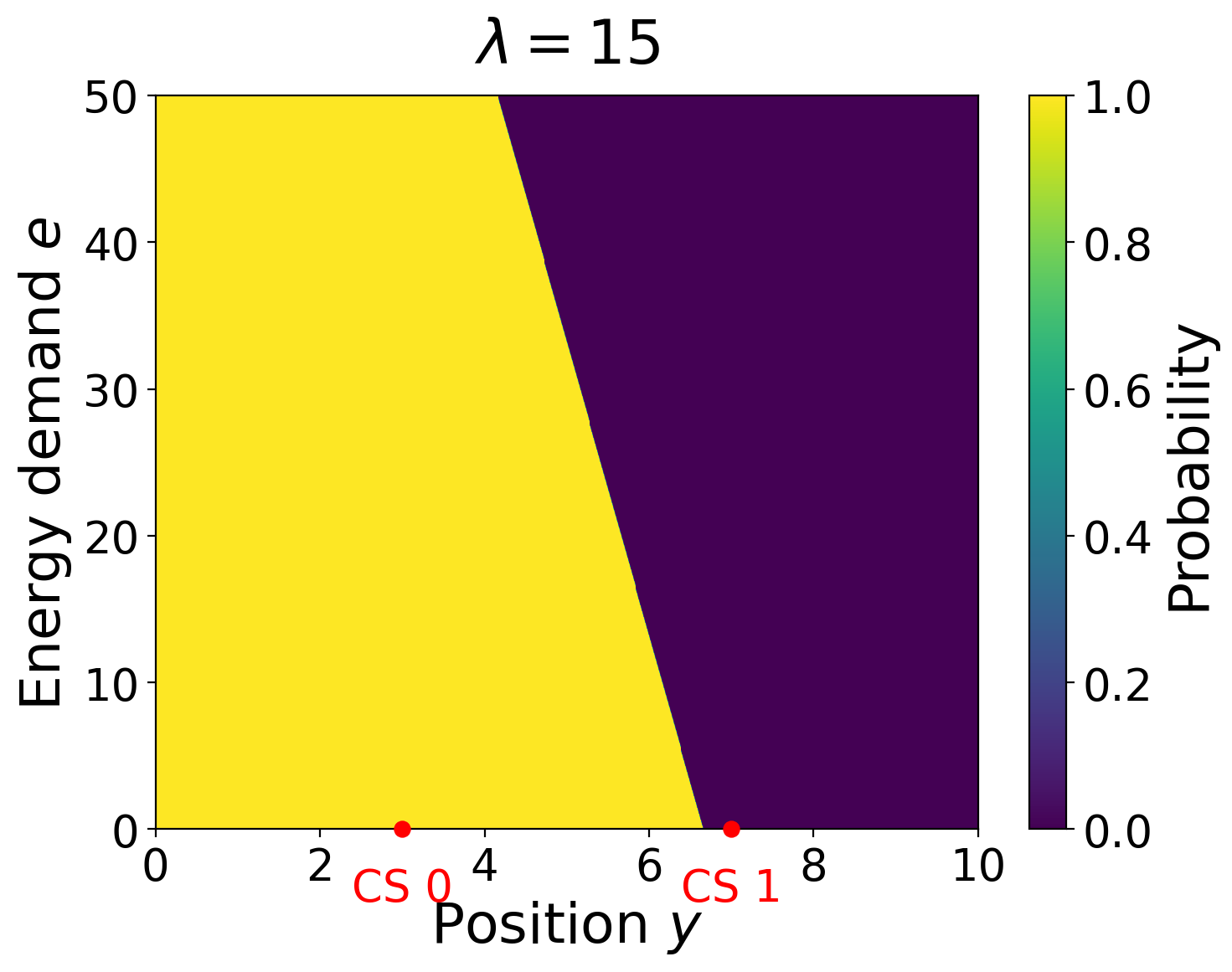}
    \captionof{figure}{\label{fig:proba_to_go_to_CS1_mass}Probability of choosing CS 0 for different values of the total flow mass $\lambda$.}
\end{center}

Note that for $\lambda=10$, the graph is the same as in Figure \ref{fig:proba_to_go_to_CS1_mass} when $\sigma=0.005$. When $\lambda=10$, we first observe that more EVs are attracted by CS 1, due to its advantageous price (recall that $r_1 = 0.35$ \euro/kWh against $r_0 = 0.5$ \euro/kWh, and all other characteristics are symmetrical between the two CSs). In particular, EVs with large energy demand prioritise CS 1, as the boundary tilts to the left as $e$ increases. However, this tendency with the congestion. Indeed, the larger $\lambda$, the further to the right the frontier between the two Voronoi cells lies: as congestion builds up at CS 1, part of the demand is redirected towards CS 0. These observations are confirmed by Figure \ref{fig:dashboard}, which displays the mass of players choosing each CS, $\int a_i \, d\mu$, together with its energy market share,
\begin{displaymath}
    \frac{\Psi_1 a_i}{(\Psi_1 a_1) +(\Psi_1 a_2)}.
\end{displaymath}

\begin{center}
    \includegraphics[scale=0.35]{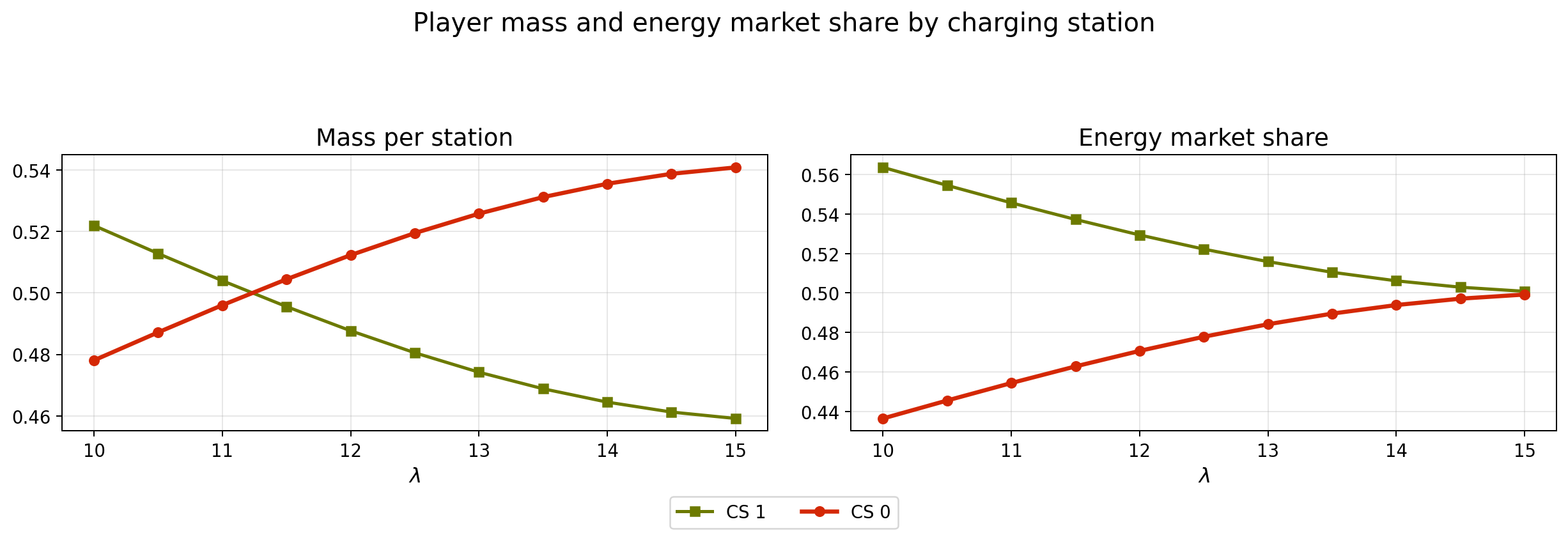}
    \captionof{figure}{\label{fig:dashboard}Mass and energy market share of each station as functions of $\lambda$.}
\end{center}

When $\lambda=10$, CS 1 attracts the majority of EVs and its market share is maximal. However, the physical capacity constraint modelled by $h$ implies that the waiting time blows up as the total mass $\lambda$ increases, driving both market shares towards 50\%. Note that, for large $\lambda$, CS0 attracts more EVs, but with smaller energy demand than those choosing CS 1, which catches a larger energy market share.
These observations further highlight the importance of incorporating users' energy demands into their strategy profiles.
Figures \ref{fig:dashboard} thus highlights the importance of incorporating users' energy demands into their strategy profiles: a CS can attract more EVs than the others but can still have less market share. 

\begin{center}
    \includegraphics[scale=0.42]{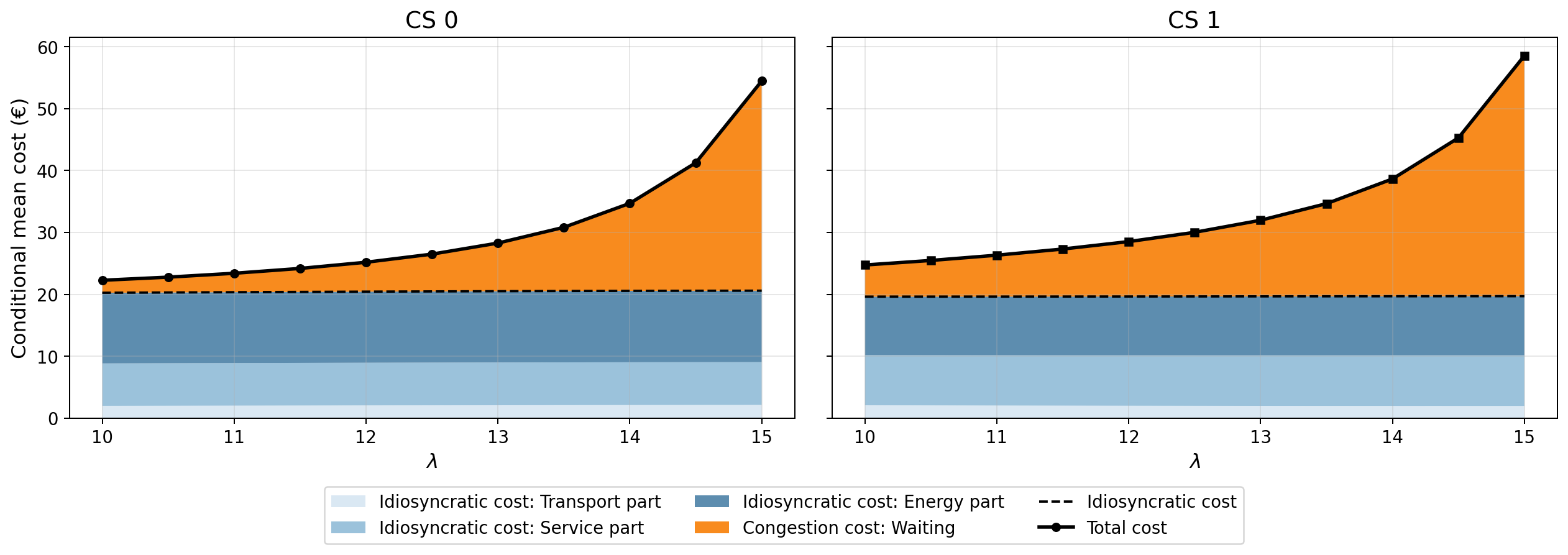}
    \captionof{figure}{\label{fig:costs} Mean cost of EVs conditionnally to choosing CS 0 (left), and CS 1 (right).}
\end{center}

Figure \ref{fig:costs} displays the decomposition of the average cost paid by the EVs choosing CS 0 or 1. The area below the dashed line represents the idiosyncratic cost, decomposed in our setting by the transport, the service and the energy cost. As expected, the two CSs exhibit different cost profiles, reflecting the different populations they attract. EVs with large energy demand make longer detours in order to charge where the price is low, i.e., at CS 1. As a consequence, both the average transport and service costs are higher at CS 1. Althouth the energy demand at CS 1 is higher in average compared to CS 0, the overall energy cost at CS 1 is lower due to the lower charging price. Overall, the idiosyncratic cost at CS 1 is lower for all values of $\lambda$.

The area above the dashed line represents the congestion cost, represented in this setting by the waiting cost. Figure \ref{fig:costs} also illustrates the sharp increase of the waiting cost as the total mass $\lambda$ approaches 2. Note that the graph does not represents the cost for $\lambda > 1.8$, since the explosion of the waiting time beyond this value would make the idiosyncratic cost negligible compared to the congestion cost.

\section{Conclusion}\label{sec:conclusion}

In this paper, we studied a class of congestion games with a continuum of
heterogeneous players and a finite set of resources, without assuming the
existence of a potential function. The proposed framework allows the
interaction cost associated with a resource to depend on general aggregate
characteristics of the players choosing it. It therefore captures situations
in which players differ not only in their individual preferences, but also in
the congestion externalities they impose on others.

We first characterized the potential structure of the game and showed that it
may fail as soon as players contribute differently to congestion. To address
the resulting non-potential setting, we introduced the notion of a Stochastic
Variational Nash Equilibrium. This equilibrium describes each player's
strategy as a function of both their characteristics and an idiosyncratic
random shock. We established that players use pure strategies after observing
the shock and showed that, under Gumbel-distributed shocks, the resulting
aggregate allocation coincides with a logit Quantal Response Equilibrium.

The stochastic regularization also provides a means of establishing the
existence of equilibrium without relying on an optimization problem. For every
positive noise parameter, and under the stated assumptions, we proved the existence of a Stochastic Variational Nash Equilibrium through the Schauder--Tychonoff fixed-point theorem. By
letting the noise parameter tend to zero, we then recovered the existence of a
pure Nash equilibrium. Under the additional coordinate-wise monotonicity condition on the
aggregate interaction costs, we further established uniqueness of both the
regularized equilibrium and the pure Nash equilibrium. Under the same additional condition, we also showed that the allocations induced by the regularized equilibria converge to the unique pure
Nash equilibrium as the noise vanishes.

The numerical application to electric vehicle charging illustrates the
relevance of accounting for heterogeneous contributions to congestion. In
this application, users differ in both their locations and their energy
demands. The simulations illustrate the convergence of the
regularized allocation towards a deterministic partition as the noise
parameter decreases. Importantly, the station attracting the largest number of users
need not serve the largest share of total energy demand. This distinction
highlights the importance of considering the composition of the population,
rather than only its mass, when modeling congestion.

Promising directions for future research include
deriving quantitative convergence rates for the zero-noise limit and to develop numerical methods
with convergence guarantees for small values of the regularization parameter.
These developments could also provide a basis for
bilevel models in which a system operator, or a set of system operators acting in a competitive environment themselves, optimizes resource characteristics such as prices, capacities, or locations while accounting for the equilibrium response of a
heterogeneous population.

\paragraph*{Acknowledgements.} The authors thank Fabien Leurent for fruitful discussions and valuable recommendations.

\paragraph*{AI usage statement.} AI was utilized during the drafting process for grammatical correction and provide a preliminary check of the mathematical consistency. 
In addition, AI helped to find the proper framework in Section \ref{subsec:non_potential_game} to obtain Proposition \ref{prop:potential_condition} by adapting the arguments developed in \cite{Vainberg1964}. The core concepts and mathematical proofs of the rest of the paper remain the original work of the authors. The authors take full responsibility for the correctness and content of the manuscript.

\appendix
\section{Proof of Theorem \ref{thm:path_characterization}}
\label{app:proof_thm_path}
\emph{Step 1: \ref{item:pot} $\Rightarrow$ \ref{item:tri}.}
Let $\mathcal G$ be a potential and fix $a,a'\in\mathcal A$. Set

\begin{displaymath}
    b=a'-a\in\mathcal H_0,\quad a_t=a+tb\quad \text{and
}\quad\varphi(t)=\mathcal G(a_t),\qquad \forall t\in[0,1].
\end{displaymath}
Let $t\in(0,1)$. Since
$a_{t+s}\in\mathcal A$ for $s\in[-t,1-t]$, both $b$ and $-b$ belong to
$\mathcal T_{\mathcal A}(a_t)$. Definition~\ref{def:potential_games}
applied in the direction $b$, respectivaly $-b$, gives that the right, respectively the left, derivative of $\varphi$
at $t$ exists and equals
\begin{displaymath}
    \lim_{s\downarrow0}\frac{\varphi(t+s)-\varphi(t)}{s}
=\langle\boldsymbol h(a_t),b\rangle_{L^2}, \text{ respectively }-\lim_{s\downarrow0}\frac{\varphi(t-s)-\varphi(t)}{s}
=-\langle\boldsymbol h(a_t),-b\rangle_{L^2}.
\end{displaymath}
Let us set $g(t) = \langle\boldsymbol h(a_t),b\rangle_{L^2}$ for any $t\in[0,1]$. Therefore, $\varphi$ is differentiable on
$(0,1)$ with $\varphi'=g$, where $g$ is continuous on $[0,1]$ by Assumption
\ref{ass:radial_continuity}. At $t=0$ (resp.\ $t=1$), the same argument in
the direction $b$ (resp.\ $-b$) shows that $\varphi$ admits a one-sided
derivative, so that $\varphi$ is continuous on $[0,1]$. The fundamental theorem of calculus yields
\begin{equation}\label{eq:increment_formula}
    \mathcal G(a')-\mathcal G(a)=\varphi(1) - \varphi(0)
    =\int_0^1 g(t)\,dt
    =\mathcal I(a\to a').
\end{equation}
The increment of $\mathcal G$ between two points thus equals the line
integral along the joining segment. Summing
\eqref{eq:increment_formula} along the edges of any closed polygonal path
telescopes to zero, which contains \eqref{eq:zero_circulation} as the
particular case $K=3$. Formula \eqref{eq:increment_formula} also proves
the last assertions of the theorem: any potential has increments
prescribed by $\mathcal I$, so that two potentials differ by an additive
constant.

\smallskip
\emph{Step 2: \ref{item:tri} $\Rightarrow$ \ref{item:lineint}.}
Fix $\bar a\in\mathcal A$ and define $\mathcal G$ by
\eqref{eq:potential_formula}. Let $a\in\mathcal A$,
$b\in\mathcal T_{\mathcal A}(a)$ and a positive $\varepsilon$, such that $a+tb\in\mathcal A$ for any
$t\in[0,\varepsilon]$. Applying \eqref{eq:zero_circulation} to the triangle
$(\bar a,\,a,\,a+tb)$ together with \eqref{eq:antisymmetry_I} gives, for
every $t\in(0,\varepsilon]$,
\begin{displaymath}
    \mathcal G(a+tb)-\mathcal G(a)
    =\mathcal I(a\to a+tb)
    =\int_0^1
    \langle\boldsymbol h(a+stb),\,tb\rangle_{L^2}\,ds
    =\int_0^t
    \langle\boldsymbol h(a+ub),\,b\rangle_{L^2}\,du .
\end{displaymath}
The map $u\mapsto\langle\boldsymbol h(a+ub),b\rangle_{L^2}$ is continuous
on $[0,\varepsilon]$ by Assumption \ref{ass:radial_continuity} applied to the pair
$(a,a+\varepsilon b)$. Dividing by $t$ and letting it vanish yield
$D^+_{\mathcal A}\mathcal G(a)[b]
=\langle\boldsymbol h(a),b\rangle_{L^2}$. Since, it can be done for any $b\in\cT(a)$, $\mathcal G$ is a
potential.

\smallskip
\emph{Step 3: \ref{item:lineint} $\Rightarrow$ \ref{item:pot}} is
immediate, which closes the cycle of implications.

\section{Proof of Proposition \ref{prop:potential_condition}}
\label{app:proof_prop_symmetry}
\emph{Necessity.}
Suppose that the game admits a potential $\mathcal G$.
Let $a\in\mathcal A$, and let $b,b'\in\mathcal T_{\mathcal A}(a)$ be jointly
feasible at $a$. There exists $\varepsilon>0$ such that
$a_{s,t}=a+sb+tb'$ belongs to $\mathcal A$ for every
$s,t\in[0,\varepsilon]$. Define
\begin{displaymath}
    \varphi(s,t)=\mathcal G(a_{s,t}),
    \qquad
    \forall(s,t)\in[0,\varepsilon]^2.
\end{displaymath}
Let $(s,t)\in(0,\varepsilon)^2$. Since $s\pm\delta,\,t\pm\delta\in
[0,\varepsilon]$ for $\delta>0$ small enough, both $b$ and $-b$ (resp.\ $b'$
and $-b'$) belong to $\mathcal T_{\mathcal A}(a_{s,t})$. Using
Definition~\ref{def:potential_games} in the directions $\pm b$ and $\pm b'$,
the one-sided derivatives from the right and from the left coincide, so
$\varphi$ admits genuine partial derivatives
\begin{displaymath}
    D_s\varphi(s,t)
    =
    \left\langle\boldsymbol h(a_{s,t}),b\right\rangle_{L^2},
    \qquad
    D_t\varphi(s,t)
    =
    \left\langle\boldsymbol h(a_{s,t}),b'\right\rangle_{L^2}.
\end{displaymath}
Since $\boldsymbol h$ is of class $C^1$ relative to $\mathcal A$, and since
the increments $a_{s,t\pm\delta}-a_{s,t}=\pm\delta b'$ correspond to
admissible points of $\mathcal A$, the map $t\mapsto\boldsymbol h(a_{s,t})$
is differentiable in $L^2$ with derivative
$D_{\mathcal A}\boldsymbol h(a_{s,t})[b']$, and similarly in $s$. Hence the
mixed partial derivatives
\begin{displaymath}
    D_tD_s\varphi(s,t)
    =
    \left\langle
        D_{\mathcal A}\boldsymbol h(a_{s,t})[b'],b
    \right\rangle_{L^2},
    \qquad
    D_sD_t\varphi(s,t)
    =
    \left\langle
        D_{\mathcal A}\boldsymbol h(a_{s,t})[b],b'
    \right\rangle_{L^2}
\end{displaymath}
exist on $(0,\varepsilon)^2$ and are continuous there, by continuity of
$D_{\mathcal A}\boldsymbol h$ in the operator norm and continuity of
$(s,t)\mapsto a_{s,t}$ in $L^2$. Schwarz's theorem then gives
\begin{displaymath}
    \left\langle
        D_{\mathcal A}\boldsymbol h(a_{s,t})[b'],b
    \right\rangle_{L^2}
    =
    \left\langle
        D_{\mathcal A}\boldsymbol h(a_{s,t})[b],b'
    \right\rangle_{L^2},
    \qquad
    (s,t)\in(0,\varepsilon)^2.
\end{displaymath}
Letting $(s,t)\to(0,0)$ and using
\begin{displaymath}
    \bigl|
    \bigl\langle
        (D_{\mathcal A}\boldsymbol h(a_{s,t})
        -D_{\mathcal A}\boldsymbol h(a))[b'],b
    \bigr\rangle_{L^2}
    \bigr|
    \le
    \bigl\|
        D_{\mathcal A}\boldsymbol h(a_{s,t})
        -D_{\mathcal A}\boldsymbol h(a)
    \bigr\|_{\mathcal L(\mathcal H_0,L^2)}
    \|b'\|_{L^2}\|b\|_{L^2},
\end{displaymath}
together with the continuity of $D_{\mathcal A}\boldsymbol h$, yields
\eqref{eq:symmetry_condition}.

\medskip
\noindent
\emph{Sufficiency.}
Fix $\bar a\in\mathcal A$ and define
$\mathcal G(a)=\mathcal I(\bar a\to a)$. We show directly that $\mathcal G$ is
a potential. Let $a\in\mathcal A$, $b\in\mathcal T_{\mathcal A}(a)$ and a positive $\varepsilon$ such that $a+sb\in\mathcal A$ for all $s\in[0,\varepsilon]$, and set
\begin{displaymath}
    b'=a-\bar a\in\mathcal H_0,
    \qquad
    c_t=\bar a+t\,b'\in\mathcal A,\quad t\in[0,1].
\end{displaymath}
For $s\in[0,\varepsilon]$ and $t\in[0,1]$, convexity gives
$c_t+ts\,b=(1-t)\,\bar a+t\,(a+sb)\in\mathcal A$, so that
\begin{displaymath}
    \mathcal G(a+sb)
    =\int_0^1
    \bigl\langle
        \boldsymbol h(c_t+ts\,b),\,b'+s\,b
    \bigr\rangle_{L^2}\,dt .
\end{displaymath}
For fixed $t$, the points $c_t+ts\,b$, $s\in[0,\varepsilon]$, all lie in
$\mathcal A$, so relative $C^1$ regularity makes the integrand
differentiable in $s$, with derivative
\begin{displaymath}
    [0,\varepsilon]\ni s\mapsto t\,\langle D_{\mathcal A}\boldsymbol h(c_t+ts\,b)[b],\,
b'+s\,b\rangle_{L^2}
+\langle\boldsymbol h(c_t+ts\,b),\,b\rangle_{L^2}.
\end{displaymath}
It is jointly continuous in $(s,t)$, hence bounded on
$[0,\varepsilon]\times[0,1]$. Differentiating under the integral sign at
$s=0^+$ therefore yields
\begin{equation}\label{eq:homotopy_derivative}
    D^+_{\mathcal A}\mathcal G(a)[b]
    =\int_0^1
    \Bigl[
    t\,\bigl\langle
        D_{\mathcal A}\boldsymbol h(c_t)[b],\,b'
    \bigr\rangle_{L^2}
    +\bigl\langle\boldsymbol h(c_t),\,b\bigr\rangle_{L^2}
    \Bigr]\,dt .
\end{equation}
Fix now $t\in(0,1)$. For $u\in[0,t\varepsilon]$ and $v\in[0,1-t]$, one has
$u/(t+v)\le\varepsilon$, so convexity gives
\begin{displaymath}
    c_t+u\,b+v\,b'
    =\bigl(1-(t+v)\bigr)\,\bar a
    +(t+v)\,\Bigl(a+\tfrac{u}{t+v}\,b\Bigr)
    \;\in\;\mathcal A .
\end{displaymath}
Hence $b$ and $b'$ are jointly feasible at $c_t$, and the symmetry
hypothesis \eqref{eq:symmetry_condition} gives
\begin{equation}\label{eq:symmetry_chord}
    \bigl\langle
        D_{\mathcal A}\boldsymbol h(c_t)[b],\,b'
    \bigr\rangle_{L^2}
    =
    \bigl\langle
        D_{\mathcal A}\boldsymbol h(c_t)[b'],\,b
    \bigr\rangle_{L^2},
    \qquad t\in(0,1).
\end{equation}
Set $\varphi(t)=\langle\boldsymbol h(c_t),b\rangle_{L^2}$ for all $t$. As in Step 1, $\varphi$
is continuous on $[0,1]$ and differentiable on $(0,1)$, with bounded
continuous derivative
$\varphi'(t)=\langle D_{\mathcal A}\boldsymbol h(c_t)[b'],b\rangle_{L^2}$.
By \eqref{eq:symmetry_chord}, the integrand of
\eqref{eq:homotopy_derivative} equals
$t\,\varphi'(t)+\varphi(t)=\frac{d}{dt}\bigl[t\,\varphi(t)\bigr]$ on $(0,1)$. Hence
\begin{displaymath}
    D^+_{\mathcal A}\mathcal G(a)[b]
    =\int_0^1\frac{d}{dt}\bigl[t\,\varphi(t)\bigr]\,dt
    =\varphi(1)
    =\bigl\langle\boldsymbol h(a),\,b\bigr\rangle_{L^2},
\end{displaymath}
so $\mathcal G$ is a potential.

\section{Well-posedness of the derivative relative to \texorpdfstring{$\mathcal{A}$}{A}}
\label{app:derivative}Definition~\ref{def:C1-relative} is in the spirit of Whitney's notion of a $C^1$ function on a closed subset of $\mathbb R^n$~\cite{Whitney1934}, and of the flat derivative on the space of probability measures \cite[Sect.~2.2]{CardaliaguetDelarueLasryLions2019}. In these notions, the first-order expansion identifies the derivative only through its pairing with increments between admissible points. Consequently, any component orthogonal to the linear span of such increments is not detected by the expansion, and the derivative may therefore fail to be uniquely determined in the ambient space. For the flat derivative, the relevant increments are signed measures with zero total mass, so additive constants are invisible. Uniqueness is then recovered by imposing a normalization convention. 

\smallskip
\noindent
In our setting, no additional normalization is required. Indeed, $D_{\mathcal A}F(a)$ is required to belong to $\cL(\mathcal H_0,E)$, and Lemma~\ref{lem:uniqueness} shows that the linear span of the admissible increments $\{b-a : a,b\in\mathcal A\}$ is dense in $\mathcal H_0$. Hence, if two elements of $\mathcal H_0$ satisfy \eqref{eq:C1-limit}, their difference is orthogonal to a dense subspace of $\mathcal H_0$, and must therefore vanish. Thus, \eqref{eq:C1-limit} determines $D_{\mathcal A}F(a)$ uniquely as an element of $\cL(\mathcal H_0,E)$.
\begin{lemma}[Uniqueness of the relative derivative]
\label{lem:uniqueness}
For every $a\in\mathcal A$, the closed linear span of
$\mathcal A-a=\{\widetilde a-a:\widetilde a\in\mathcal A\}$ in
$L^2(X,\mu;\mathbb R^N)$ equals $\mathcal H_0$. Consequently, the map
$D_{\mathcal A}F$ in Definition~\ref{def:C1-relative} is uniquely determined
by \eqref{eq:C1-limit}.
\end{lemma}

\begin{proof}
Since $\operatorname{span}(\mathcal A-a)\supseteq\mathcal A-\mathcal A$, it
suffices to show that $\mathcal A-\mathcal A$ spans a dense subspace of
$\mathcal H_0$. Let $(e_1,\ldots,e_N)$ denote the canonical basis of
$\mathbb R^N$. For $i\neq j$ and a measurable set $S\subseteq X$, the
allocations
\begin{displaymath}
    \widetilde a^{(1)}:=e_i\,\mathbf 1_S+e_1\,\mathbf 1_{S^c},
    \qquad
    \widetilde a^{(2)}:=e_j\,\mathbf 1_S+e_1\,\mathbf 1_{S^c}
\end{displaymath}
belong to $\mathcal A$, whence
$(e_i-e_j)\,\mathbf 1_S\in\mathcal A-\mathcal A$. Linear combinations of such
elements are exactly the simple functions with values in
$V_0:=\{v\in\mathbb R^N:\sum_{k=1}^N v_k=0\}
=\operatorname{span}\{e_i-e_j: i\neq j\}$, and these are dense in
$L^2(X,\mu;V_0)=\mathcal H_0$.
\end{proof}

\begin{remark}
If $F$ is of class $C^1$ relative to $\mathcal A$, then for every
$a\in\mathcal A$ and every feasible direction
$b\in\mathcal T_{\mathcal A}(a)$,
\begin{displaymath}
    \lim_{t\downarrow0}
    \frac{F(a+tb)-F(a)}{t}
    =
    D_{\mathcal A}^+F(a)[b],
\end{displaymath}
as follows from \eqref{eq:C1-limit} applied to $\widetilde a=a+tb$,
$t\downarrow 0$. Thus, the relative derivative represents all one-sided
feasible directional derivatives of $F$.
\end{remark}

\bibliographystyle{plain}
\bibliography{biblio}
\end{document}